\documentclass{article}
\usepackage[utf8]{inputenc}
\usepackage{graphicx,epstopdf}
\usepackage{amsmath,amssymb,amsthm}

\usepackage{dsfont}
\usepackage{hyperref}
\usepackage{subfigure}
\usepackage{epstopdf}
\usepackage{geometry}
\usepackage{indentfirst}
\usepackage{verbatim}
\usepackage{cite}
\allowdisplaybreaks[1]
\numberwithin{equation}{section}
\newtheorem{theorem}{Theorem}[section]
\newtheorem{lemma}{Lemma}[section]

\def\0{{\mathbf{0}}}
\def\u{{\mathbf{u}}}
\def\ul{\underline{u}}
\def\vl{\underline{v}}

\def\0{{\mathbf{0}}}

\title{Analysis of A New Adaptive Time Filter Algorithm for The Unsteady Stokes/Darcy Model
\thanks{Subsidized by NSFC (grant No.12001347, No.12101494 and No.11771259),
Scientific Research Program Funded by Shaanxi Provincial Education Department (Program No. 21JP019 and No. 21JK0935) and
the Natural Science Foundation of Shaanxi Province(No. 2021JQ-426).
Thanks for the support from special support programme to develop innovative talents in the region of Shaanxi province and youth innovation team on computationally efficient numerical methods based on new energy problems in Shaanxi province.}}
\author{
Yi Qin\thanks{School of Mathematics and Data Science, Shaanxi University of Science and Technology, Xi'an, Shaanxi 710021, China. ({\tt 4545@sust.edu.cn})},
Yang Wang\thanks{School of Mathematics and Data Science, Shaanxi University of Science and Technology, Xi'an, Shaanxi 710021, China. ({\tt 210911017@sust.edu.cn})},
Yi Li\thanks{School of Mathematics, Northwest University, Xi'an, Shaanxi 710127, China. ({\tt liyizz@nwu.edu.cn})}
and Jian Li\thanks{School of Mathematics and Data Science, Shaanxi University of Science and Technology, Xi'an, Shaanxi 710021, China. ({\tt jianli@sust.edu.cn})}
}
\date{}
\begin{document}
\maketitle
\begin{abstract}
In this report, we propose a new adaptive time filter algorithm for the unsteady Stokes/Darcy model.
First we present a first order $\theta$-scheme with the variable time step which is one parameter family of Linear Multi-step methods
and use a time filter algorithm to increase the convergence order to second order with almost no increasing the amount of computation.
Furthermore, we construct coupled and decoupled adaptive algorithms.
Then we analyze stabilities and the second-order accuracy of variable time-stepping algorithm for Linear Multi-step methods plus time filter, respectively.
Finally, we use two numerical experiments to verify theoretical results including effectiveness, convergence and efficiency with adaptive method.
\end{abstract}


\section{Introduction}
In recent years, the coupling problem between free fluid flow and porous media flow can effectively describe many problems, such as pollution of surface water and groundwater, oil exploitation, industrial fluid filtration and blood movement, so the coupling problem has been studied by more and more scholars. In this paper, we study this coupling problem based on an important model, namely Stokes/Darcy model. We consider that it is controlled by Stokes equation in the free fluid flow region and Darcy equation in the porous media region.

A lot of work has been done on the Stokes/Darcy model. Numerical methods for steady Stokes/Darcy model include finite element methods, discontinuous Galerkin methods, interface relaxation methods, Lagrange multiplier methods, two-grid or multi-grid methods, domain decomposition methods\cite{ACSD,OSCS,FRAS,DJAM,PDGM,ACNS,AANF,SCPW,ARPE,FFSD,CFFW,SNSM,RDDM,APRD,OEEO,FTMS,ATFC} and so on.
However, for the unsteady Stokes/Darcy model, The discretization of time is still a problem that needs to be studied.
Many scholars use first-order algorithms that are more computationally efficient and easy to implement\cite{ATDN,MMXZ,NATG,PTMF,ADST}, and some scholars use higher-order algorithms with higher precision\cite{AELT,PTSM,APMW}.
At present, the research on constant time step are relatively mature, many scholars have begun to notice that the variable time-stepping algorithm and construct corresponding adaptive algorithm which has many advantages in both time accuracy and computational efficiency.
The constant time-stepping algorithm fixes the time step in the process of calculation, and cannot adjust the size of the time step according to the actual situation. Compared with it, adaptive algorithm can adjust time step size automatically according to the needs of different models, and shorten the step as much as possible to ensure the computational efficiency.

In this paper, we first give a first order $\theta$-scheme with the variable time step, which is a parameter family of Linear Multi-step methods for the unsteady Stokes/Darcy model. %Compared with the Backward Euler method, it has higher accuracy.
In particular, when $\theta=0$, it's the Backward Euler method, when $\theta=1/2$, it's the Crank-Nicolson method, and when $\theta=1$, it's the Forward Euler method.
Here we consider the more general case, which is $0<\theta<1/2$.
Since time filter are easy to modify and implement programmatically and can improve the accuracy of algorithms, time filter are widely used\cite{AGWJ,APMW,PTSM,AEVS,APMF}.
So based on the first order Linear Multi-step methods, we think about the effect of adding the simple time filter for the unsteady Stokes/Darcy model. The method is modular and need to add only two additional line of code, which increases the accuracy of the Linear Multi-step Method from first to second order.
We propose variable time-stepping algorithms for coupled and decoupled Linear Multi-step methods plus time filter, and construct corresponding adaptive algorithms. The stabilities and the second-order accuracy are analyzed and we can find that the results do not change as the step size increases or decreases.
Finally, we make two numerical experiments. In the first test, we verify the stabilities of the variable time-stepping algorithms by three sets of different variations in time steps. In the second test, we show that convergence order and CPU time of coupled and decoupled adaptive algorithms are increased from the first order to the second order, and by comparing coupled and decoupled algorithms, we get that the decoupled algorithm is more computationally efficient.

The rest of this paper are as follows:
in Section \ref{2}, we review coupled Stokes/Darcy model and weak formulation. %related boundary conditions and interface conditions in detail, and define the corresponding spaces and norms for subsequent analysis, and introduce the properties and inequalities needed in the analysis process.
Section \ref{3} is divided into two small parts, one is to introduce variable time-stepping algorithms of the coupled and decoupled Linear Multi-step methods plus time filter, at the same time, we construct the adaptive algorithms. The other is the stability analysis for the variable time-stepping algorithm.
Section \ref{4}, we give the error estimates of the two variable time-stepping algorithms respectively.
We use two numerical experiments to verify the effectiveness, convergence and efficiency of adaptive algorithm in Section \ref{5}.
\section{The Stokes/Darcy model and weak formulation\label{2}}
This section, the coupled Stokes/Darcy model is considered in a bounded smooth domain $\Omega\subset {\bf R}^d$, $d=2$ or $3$, which consists of a free fluid flow region $\Omega_f$ and a porous media flow region $\Omega_p$ with the unit outward normal vectors $\vec{n}_f$ and $\vec{n}_p$ on $\partial\Omega_f$ and $\partial\Omega_p$, where $\Omega_f$ and $\Omega_p$ are two disjoint, connected and bounded domains.
The interface $\Gamma=\overline{\Omega}_f\cap\overline{\Omega}_p$ separated the two regions $\Omega_f$ and $\Omega_p$, we need to pay attention to $\vec{n}_f=-\vec{n}_p$ on $\Gamma$ and define $\Gamma_i=\partial\Omega_i\cap\partial\Omega$, for $i$ $=$ $f,$ $p.$ We can refer to the sketch (Figure \ref{fig1}).

\begin{figure}[ht]
    \centering
    \includegraphics[width=7cm]{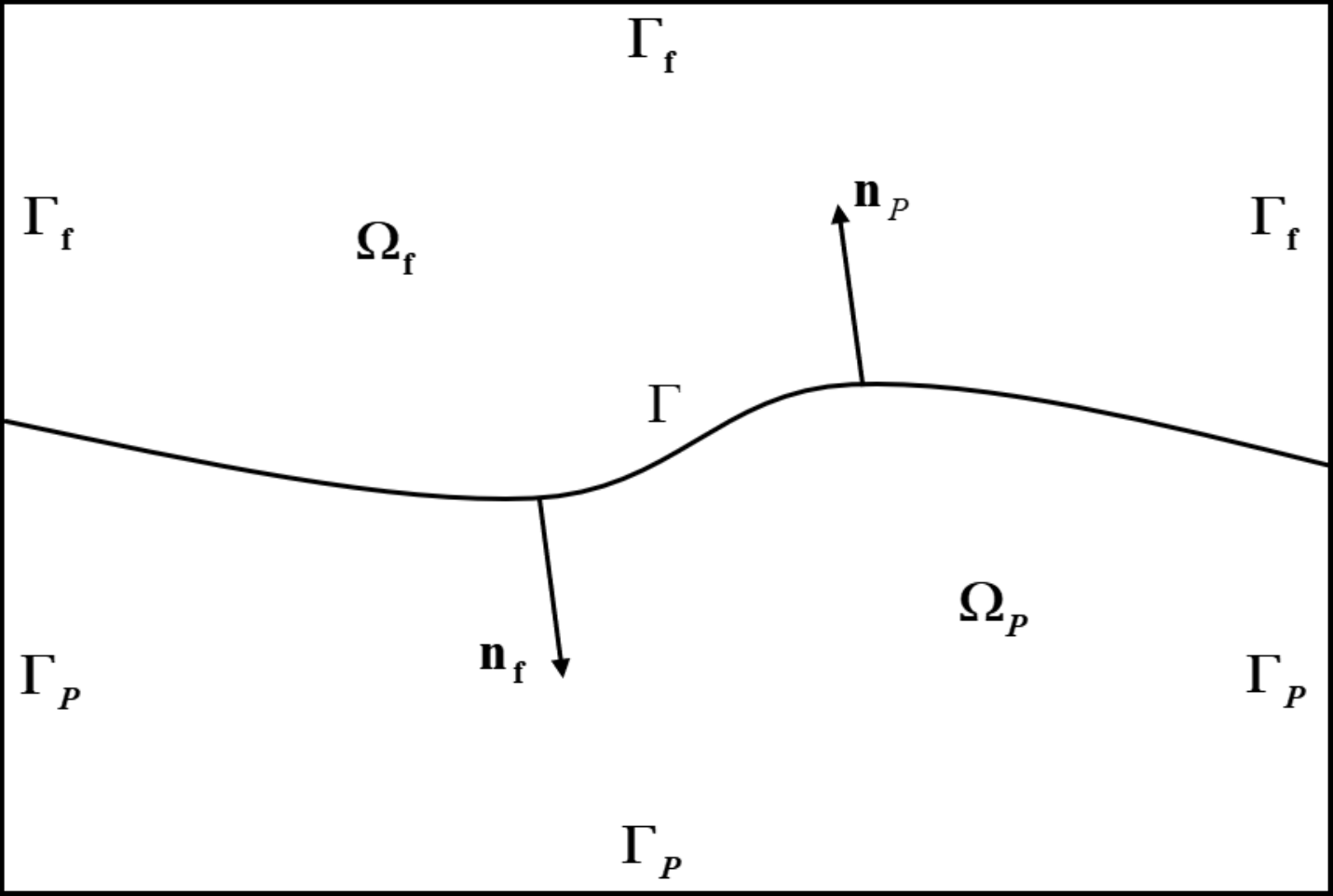}
    \caption{\small A bounded smooth domain $\Omega$ consisting of a fluid flow region $\Omega_f$ and a porous media flow region $\Omega_p$ separated by an interface $\Gamma$.}\label{fig1}
\end{figure}

For the finite time interval $[0,T]$. The flow in the free fluid flow region $\Omega_f$ we describe it using the Stokes equation, which is stated as follows: for fluid velocity $\vec{u}_f(\vec{x},t)$ and kinematic pressure $p_f(\vec{x},t)$
\begin{eqnarray}
&\label{stokes1}\frac{\partial\vec{u}_f}{\partial t}-\nu\triangle\vec{u}_f+\nabla p_f=\vec{g}_f, \ \quad\qquad &\mbox{in}\;\Omega_f\times(0,T],\\
&\label{stokes2}\nabla\cdot \vec{u}_f=0, \qquad\qquad\qquad\qquad\qquad&\mbox{in}\;\Omega_f\times(0,T],\\
&\label{initial}\vec{u}_f(\vec{x},0)=\vec{u}_f^0(\vec{x}),\quad\qquad\qquad\qquad&\mbox{in}\;\Omega_f,
\end{eqnarray}
where $\nu>0$ is the kinetic viscosity and $\vec{g}_f(\vec{x},t)$ is the external force.

The flow in the porous media region $\Omega_p$ we describe it using the follow equation:
\begin{eqnarray}
&\label{Darcy1}S_0\frac{\partial\phi_p}{\partial t}+\nabla\cdot\vec{u}_p=g_p,
\qquad\qquad\qquad &\mbox{in}\;\Omega_p\times(0,T],\\
&\label{Darcy2}\vec{u}_p=-\mathds{K}\nabla\phi_p,
\qquad\qquad\qquad\qquad\quad&\mbox{in}\;\Omega_p\times(0,T],\\
&\label{initia2}\phi_p(\vec{x},0)=\phi_p^0(\vec{x}),\qquad\qquad\qquad\qquad&\mbox{in}\;\Omega_p.
\end{eqnarray}
Combining the equation (\ref{Darcy1}) and (\ref{Darcy2}), we have the Darcy equation: for the piezometric (hydraulic) head $\phi_p(\vec{x},t)$
\begin{eqnarray}\label{darcy}
S_0\frac{\partial\phi_p}{\partial t}-\nabla\cdot(\mathds{K}\nabla\phi_p)=g_p,\qquad\mbox{in}\;\Omega_p\times(0,T],
\end{eqnarray}
where $\vec{u}_p$ is the flow velocity in the porous media region which is proportional to the gradient of $\phi_p$, namely, the Darcy's law. $S_0$ is the specific mass storativity coefficient and $\phi_p=z+\frac{p_p}{\rho g}$, here $z$ and $p_p$ denote the relative depth from a fixed reference level and the dynamic pressure, $\rho$ and $g$ represent the density and the gravitational constant, respectively. $g_p(\vec{x},t)$ is a source term and $\mathds{K}=\{\mathds{K}_{jj}\}_{d\times d}$ denotes a symmetric and positive definite matrix with the smallest eigenvalue $\mathds{K}_{min}>0$, which is allowed to vary in space.

We usually assume that the fluid velocity $\vec{u}_f(\vec{x},t)$ and the piezometric (hydraulic) head  $\phi_p(\vec{x},t)$ satisfy the homogeneous Dirichlet boundary conditions:
\begin{equation}
\begin{split}\label{boundary}
\vec{u}_f=0,\qquad\mbox{on}\;\Gamma_f\times(0,T]\qquad\text{and}\qquad
\phi_p=0,\qquad\mbox{on}\;\Gamma_p\times(0,T].
\end{split}
\end{equation}
Interface conditions are important in the Stokes/Darcy model, which include
the conservation of mass, the balance of normal forces, and the Beaver-Joseph Saffmann conditions on $\Gamma$:
\begin{eqnarray}
\label{conservation}&&\vec{u}_f\cdot \vec{n}_f+\vec{u}_p\cdot \vec{n}_p=0,\\
\label{balance}&&p_f-\nu\vec{n}_f\frac{\partial\vec{u}_f}{\partial\vec{n}_f}=g\phi_p,\\
\label{bj}&&-\nu\vec{\tau}_i\frac{\partial\vec{u}_f}{\partial\vec{n}_f}=\frac{\alpha\nu\sqrt{d}}{\sqrt{\mbox{trace}(\mathbf{\Pi})}}\vec{u}_f\cdot\vec{\tau}_i ,\quad i=1,2,\cdots,d-1,
\end{eqnarray}
where $\vec{\tau}_i$, $i=1,2,\cdots,d-1$, are the orthonormal tangential unit vectors on the interface $\Gamma$, $d$ is the space dimension, $\alpha$ is an positive parameter and the permeability $\mathbf{\Pi}$ has the relation $\mathbf{\Pi}=\frac{\mathds{K}\nu}{g}$.

Now we give the Hilbert space that needs to be used in the next analysis process:
\begin{eqnarray*}
&&X_f=\{\vec{v}_f \in (H^1(\Omega_f)^d): \vec{v}_f=0,\ \text{on}\ \Gamma_f\},\\
&&X_p=\{\psi_p\in (H^1(\Omega_p)^d): \psi_p=0,\ \text{on}\ \Gamma_p\},\\
&&Q_f=L^2(\Omega_f),\\
&&X=X_f\times X_p.
\end{eqnarray*}
In addition, we define $X'$, $X_f'$ and $X_p'$ to represent the dual spaces of $X$, $X_f$ and $X_p$, respectively.
For the domain $D$, we define the scalar inner product in $D=\Omega_f$ or $\Omega_p$ by $(\cdot,\cdot)_D$.
For the Hilbert space $X$, $X_f'$ and $X_p'$, we denote the corresponding norms:
\begin{eqnarray*}
&\|\vec{\ul}\|_0=\sqrt{(\vec{u}_f,\vec{u}_f)_{\Omega_f}+gS_0(\phi_p,\phi_p)_{\Omega_p}},
\qquad\qquad\qquad\quad &\forall\ \vec{\ul}=(\vec{u}_f,\phi_p)\in X,\\
&\|\vec{\ul}\|_{X}=\sqrt{\nu(\nabla\vec{u}_f,\nabla\vec{u}_f)_{\Omega_f}+g\mathds{K}(\nabla\phi_p,\nabla\phi_p)_{\Omega_p}},
  \qquad\quad               &\forall\ \vec{\ul}=(\vec{u}_f,\phi_p)\in X,\\
&\Vert\vec{u}_f\Vert_f=\Vert\vec{u}_f\Vert_{L^2(\Omega_f)},
\quad\Vert\vec{u}_f\Vert_{X_f}=\Vert\nu^{\frac{1}{2}}\nabla\vec{u}_f\Vert_{L^2(\Omega_f)},
  \quad \                &\forall\ \vec{u}_f\in X_f,\\
&\Vert\phi_p\Vert_p=\Vert\phi_p\Vert_{L^2(\Omega_p)},
\quad\ \Vert\phi_p\Vert_{X_p}=\Vert\mathds{K}^{\frac{1}{2}}\nabla\phi_p\Vert_{L^2(\Omega_p)},
  \quad \                     &\forall\ \phi_p\in X_p,
\end{eqnarray*}
where the norms $\Vert\cdot\Vert_{X_f/X_p}$ and $\Vert\cdot\Vert_{f/p}$ denote
$H^1(\Omega_{f/p})$ and $L^2(\Omega_{f/p})$.
Then for the function $v(x,t)$, we define the norms:
\begin{eqnarray*}
\|v\|_{L^2(0,T;L^2)}:=(\int_{0}^{T}\|v(\cdot,t)\|_{L^2}^2dt)^{\frac{1}{2}},\quad \|v\|_{L^2(0,T;L^{\infty})}:=ess\sup_{[0,T]}\|v(\cdot,t)\|_{L^2}.
\end{eqnarray*}

Based on the above related concepts, we have weak formulations of the unsteady coupled Stokes/Darcy model(\ref{stokes1})-(\ref{bj}), which is expressed as: $\vec{g}_f\in L^2(0,T; {L}^2(\Omega_f))$ and $g_p\in L^2(0,T;  {L}^2(\Omega_p))$, find $\vec{\ul}=(\vec{u}_f, \phi_p) \in (L^2(0,T; X_f)\cap L^\infty(0,T; L^2(\Omega_f))\times L^2(0,T; X_p)\cap L^\infty(0,T; L^2(\Omega_p)) )$ and $p_f\in L^2(0,T; Q_f)$ such that $\forall$ $(\vec{\vl},q_f) \in X\times Q_f$
\begin{eqnarray}\label{coupled}
&&(\frac{\partial\vec{\ul}}{\partial t},\vec{\vl})+a(\vec{\ul},\vec{\vl})+b(\vec{\vl},p_f)=<\vec{F},\vec{\vl}>_{X'},\nonumber\\
&&b(\vec{\ul},q_f)=0,\\
&&\vec{\ul}(\vec{x},0)=\vec{\ul}^0,\nonumber
\end{eqnarray}
where
\begin{eqnarray*}
    &&(\frac{\partial\vec{\ul}}{\partial t},\vec{\vl})=(\frac{\partial\vec{u}_f}{\partial t},\vec{v}_f)+g(\frac{S_0\partial\phi_p}{\partial t},\psi_p),\\
	&&a(\vec{\ul},\vec{\vl})=a_{\Omega}(\vec{\ul},\vec{\vl})+a_\Gamma(\vec{\ul},\vec{\vl}),\\
    &&a_{\Omega}(\vec{\ul},\vec{\vl})=a_{\Omega_f}(\vec{u}_f,\vec{v}_f)+a_{\Omega_p}(\phi_p,\psi_p),\\
    &&a_{\Omega_f}(\vec{u}_f,\vec{v}_f)=\nu(\nabla(\vec{u}_f),\nabla(\vec{v}_f))_{\Omega_f}
    +\sum_{i=1}^{d-1}\int_{\Gamma}\frac{\alpha\nu\sqrt{d}}{\sqrt{\mbox{trace}(\mathbf{\Pi})}}(\vec{u}_f\cdot\vec{\tau}_i)(\vec{v}_f\cdot\vec{\tau}_i),\\
    %(\displaystyle{\frac{\alpha\nu\sqrt{d}}{\sqrt{\mbox{trace}(\mathbf{\Pi})}}}P_\t(\vec{u}_f,\vec{v}_f))_\Gamma,\\
    &&a_{\Omega_p}(\phi_p,\psi_p)=g(\mathds{K}\nabla\phi_p,\nabla\psi_p)_{\Omega_p},\\
	&&a_\Gamma(\vec{\ul},\vec{\vl})=c_{\Gamma}(\vec{v}_f,\phi_p)-c_{\Gamma}(\vec{u}_f,\psi_p)=
       g(\phi_p,\vec{v}_f\cdot \vec{n}_f)_\Gamma-g(\psi_p,\vec{u}_f\cdot \vec{n}_f)_\Gamma,\\
    &&b(\vec{\vl},p_f)=-(p_f,\nabla\cdot \vec{v}_f)_{\Omega_f},\\
	&&<\vec{F},\vec{\vl}>_{X'}=(\vec{g}_f,\vec{v}_f)_{\Omega_f}+g(g_p,\psi_p)_{\Omega_p}.
\end{eqnarray*}
%\begin{eqnarray*}
%&&(\frac{\partial\vec{\ul}}{\partial t},\vec{\vl})=(\frac{\partial\vec{u}_f}{\partial t},\vec{v}_f)+g(\frac{S_0\partial\phi_p}{\partial t},\psi_p),\\
%&&a(\vec{\ul},\vec{\vl})=a_{\Omega}(\vec{\ul},\vec{\vl})+a_\Gamma(\vec{\ul},\vec{\vl}),\\
%&&a_{\Omega}(\vec{\ul},\vec{\vl})=a_{\Omega_f}(\vec{u}_f,\vec{v}_f)+a_{\Omega_p}(\phi_p,\psi_p),\\
%&&a_{\Omega_f}(\vec{u}_f,\vec{v}_f)=\nu(\nabla(\vec{u}_f),\nabla(\vec{v}_f))_{\Omega_f}
%+\sum_{i=1}^{d-1}\int_{\Gamma}\frac{\alpha\nu\sqrt{d}}{\sqrt{\mbox{trace}(\mathbf{\Pi})}}(\vec{u}_f\cdot\vec{\tau}_i)(\vec{v}_f\cdot\vec{\tau}_i)\\
%%(\displaystyle{\frac{\alpha\nu\sqrt{d}}{\sqrt{\mbox{trace}(\mathbf{\Pi})}}}P_\t(\vec{u}_f,\vec{v}_f))_\Gamma,\\
%&&a_{\Omega_p}(\phi_p,\psi_p)=g(\mathds{K}\nabla\phi_p,\nabla\psi_p)_{\Omega_p},\\
%&&a_\Gamma(\vec{\ul},\vec{\vl})=c_{\Gamma}(\vec{v}_f,\phi_p)-c_{\Gamma}(\vec{u}_f,\psi_p)=
%       g(\phi_p,\vec{v}_f\cdot \vec{n}_f)_\Gamma-g(\psi_p,\vec{u}_f\cdot \vec{n}_f)_\Gamma,\\
%&&b(\vec{\vl},p_f)=-(p_f,\nabla\cdot \vec{v}_f)_{\Omega_f},\\
%&&<\vec{F},\vec{\vl}>_{X'}=(\vec{g}_f,\vec{v}_f)_{\Omega_f}+g(g_p,\psi_p)_{\Omega_p},
%\end{eqnarray*}
The coupled Stokes-Darcy model is well-posedness, which we can find in the other papers, we mainly analyze its numerical solution in this paper. For bilinear form $a(\cdot,\cdot)$, it is continuous and coercive:
\begin{eqnarray}\label{bilinearform}
\begin{split}
	&a(\vec{\ul},\vec{\vl})\leq C_{con} \|\vec{\ul}\|_X\|\vec{\vl}\|_X,&\qquad  &\forall\ \vec{\ul},\ \vec{\vl}\in X,\\
	&a(\vec{\ul},\vec{\ul})\geq C_{coe} \|\vec{\ul}\|_X^2,&  &\forall\ \vec{\ul}\in X,\\
    &a_{\Omega_f}(\vec{u}_f,\vec{u}_f)\geq \hat{C}_{coe} \|\vec{u}_f\|_{X_f}^2,&  &\forall\ \vec{u}_f\in X_f,\\
    &a_{\Omega_p}(\phi_p,\phi_p)\geq  \check{C}_{coe} \|\phi_p\|_{X_p}^2,&  &\forall\ \phi_p \in X_p.
\end{split}
\end{eqnarray}
At the same time, for the interface term $a_\Gamma(\cdot,\cdot)$, it satisfies the anti-symmetric properties:
\begin{eqnarray}\label{agamma}
\begin{split}
&a_\Gamma(\vec{\ul},\vec{\vl})=-a_\Gamma(\vec{\vl},\vec{\ul})\quad&\text{and}\quad&a_\Gamma(\vec{\ul},\vec{\ul})=0,&\quad&\forall\ \vec{\ul},\vec{\vl}\in X,\\
&a_\Gamma(\vec{u}_f,\vec{v}_f)=-a_\Gamma(\vec{v}_f,\vec{u}_f)\quad&\text{and}\quad&a_\Gamma(\vec{u}_f,\vec{u}_f)=0,&\quad&\forall\ \vec{u}_f,\vec{v}_f\in X_f,\\
&a_\Gamma(\phi_p,\psi_p)=-a_\Gamma(\psi_p,\phi_p)\quad&\text{and}\quad&a_\Gamma(\phi_p,\phi_p)=0,&\quad&\forall\ \phi_p,\psi_p\in X_p.
\end{split}
\end{eqnarray}

Then, we use the finite element methods(FEMs) to discretize the Stokes-Darcy model in space. Assuming $h$ is an any given small positive parameter, the regular triangulations ${\cal T}_h$, ${\cal T}_{fh}$ and ${\cal T}_{ph}$ are regular partition of triangular or quadrilateral elements of $\Omega$, $\Omega_f$ and $\Omega_p$. In order to facilitate our later analysis, we assume the domain $\Omega=\Omega_f\times\Omega_p$ is smooth enough. And we choose the Tayor-Hood elements(P2-P1) ${X}_{fh}\subset{X}_f$, $Q_{fh}\subset Q_f$  and the continuous piecewise quadratic  elements(P2) $X_{ph}\subset X_p$, which are finite element spaces, we denote $X_h={X}_{fh}\times {X}_{ph}$ and assume  that the fluid velocity space ${X}_{fh}$ and the pressure space $Q_{fh}$ satisfy the discrete LBB condition:
\begin{equation}\label{dlbb}
\inf\limits_{q_{f}^{h}\in Q_{fh}}
\sup\limits_{\vec{v}_{f}^{h}\in {X}_{fh}}\frac{(q_{f}^{h},\nabla\cdot \vec{v}_{f}^{h})_{\Omega_f}}
{\|q_{f}^{h}\|_{Q_f}\,\|\vec{v}_{f}^{h}\|_{{X}_f}}\geq \beta.
\end{equation}
We define the linear projection operator (see \cite{MMXZ}): for $\forall$ $\vec{\vl}^h\in X_h$, $q_{f}^{h}\in Q_{fh}$ and $t\in(0,T]$, $P_h:(\vec{\ul}(t),p_f(t)) \in X\times Q_f\to(P_h^{\vec{\ul}}\vec{\ul}(t), P_h^{p_f}p_f(t)) \in X_h\times Q_{fh}$ satisfies
\begin{eqnarray}\label{projection}
&&a(P_h^{\vec{\ul}}\vec{\ul}(t),\vec{\vl}^h)+b(\vec{\vl}^h,P_h^{p_f}p_f(t))=a(\vec{\ul}(t),\vec{\vl}^h)+b(\vec{\vl}^h,p_f(t)),\\
&&b(P_h^{\vec{\ul}}\vec{\ul}(t),q_{fh})=0.\nonumber
%&&a(P_h^{\u_{fh}}\u_{fh}(t),\v_{fh})+b(\v_{fh},P_h^{p_f}p_f(t))=a(\u_{fh}(t),\v_{fh})+b(\v_{fh},p_f(t)),\nonumber\\
%&&b(P_h^{\u_{fh}}\u_{fh}(t),q_{fh})=0,\\
%&&a(P_h^{\phi_{ph}}\phi_{ph}(t),\psi_{ph})+b(\psi_{ph},P_h^{p_f}p_f(t))=a(\phi_{ph}(t),\psi_{ph})+b(\psi_{ph},p_f(t)),\nonumber\\
%&&b(P_h^{\phi_{ph}}\phi_{ph}(t),q_{fh})=0.\nonumber
\end{eqnarray}
Then we assume that $(\vec{\ul}(t),p_f(t))$ is smooth enough and the projection operator $(P_h^{\vec{\ul}}\vec{\ul}(t)$,
$P_h^{p}p_f(t))$ of $(\vec{\ul}(t),p_f(t))$ satisfies the approximation properties:
\begin{eqnarray}\label{app}
&&\|P_h^{\vec{\ul}}\vec{\ul}(t)-\vec{\ul}(t)\|_0\leq Ch^2\|\vec{\ul}(t)\|_{H^2},\nonumber\\
&&\|P_h^{\vec{\ul}}\vec{\ul}(t)-\vec{\ul}(t)\|_X\leq Ch\|\vec{\ul}(t)\|_{H^2},\\
&&\|P_h^{p}p_f(t)-p_f(t)\|_{L^2}\leq Ch\|p_f(t)\|_{H^1}.\nonumber
\end{eqnarray}
In addition, we give several inequalities, including the Poincar$\acute{e}$, trace, Sobolev and inverse inequalities: there exist constants $C_i$ and $\tilde{C}_i$ $(i= p, t, s, I)$
such that for $\forall$ $\vec{v}_f\in{X}_f$ and $\psi_p\in X_p$,
\begin{eqnarray}
\begin{split}
\label{p}
&\|\vec{v}_f\|_{f}\leq C_p\|\nabla\vec{v}_f\|_{f},&\quad
&\|\psi_p\|_{f}\leq \tilde{C}_p\|\nabla\psi_p\|_{p},\\
\label{t}
&\|\vec{v}_f\|_{\Gamma}\leq C_t\|\nabla\vec{v}_f\|_{f},&\quad
&\|\psi_p\|_{\Gamma}\leq \tilde{C}_t\|\nabla\psi_p\|_{p},\\
\label{s}
&\|\vec{v}_f\|_{H^\frac{1}{2}(\partial\Omega_f)}\leq C_s\|\nabla\vec{v}_f\|_{f},&\quad
&\|\psi_p\|_{H^\frac{1}{2}(\partial\Omega_p)}\leq \tilde{C}_s\|\nabla\psi_p\|_{p},\\
\label{inverse}
&\|\vec{v}_{f}^{h}\|_{X_f}\leq C_I h^{-1}\|\vec{v}_{f}^{h}\|_{f},&\quad
&\|\vec{\psi}_{p}^{h}\|_{X_p}\leq \tilde{C}_I h^{-1}\|\vec{\psi}_{p}^{h}\|_{p}.
\end{split}
\end{eqnarray}
Note that  $C_i$ $(i= p, t, s, I)$ depend on the fluid flow domain $\Omega_f$ and $\tilde{C}_i$ $(i= p, t, s, I)$ depend on the porous media domain $\Omega_p$.
\section{Numerical algorithms and stabilities \label{3}}
We divide this section into two parts, the first part will give the variable time-stepping algorithms of coupled and decoupled Linear Multi-step methods plus time filter and construct the adaptive algorithm. And the second part will analyze stabilities of the two algorithms separately.

Before analyzing, we need to recall several lemmas that they will use multiple times during the analysis.
%A two-step method for solving ODEs consists in finding a sequence $(y_n)$, defined by:$$\alpha_0y_n+\alpha_1y_{n+1}+\alpha_2y_{n+2}=k(\beta_0f(t_n,y_n)+\beta_1f(t_{n+1},y_{n+1})+\beta_2f(t_{n+2},y_{n+2})),$$where $\alpha_i$ and $\beta_i$ are real parameters satisfying $\alpha_2\neq 0$ and $|\alpha_0|+|\beta_0|\neq 0$.

%It is convenient to associate with the two-step method a pair of polynomials $(\rho,\sigma)$ of degree 2 defined by:
%\begin{eqnarray}\label{zeta}
%\rho(\zeta)=\alpha_2\zeta^2+\alpha_1\zeta+\alpha_0,\quad\sigma(\zeta)=\beta_2\zeta^2+\beta_1\zeta+\beta_0
%\end{eqnarray}
%with
%\begin{eqnarray*}
%&&\alpha_0=-1+\alpha_2,\quad\qquad \alpha_1=1-2\alpha_2,\\
%&&\beta_0=\frac{1}{2}-\alpha_2+\beta_2,\quad\ \beta_1=\frac{1}{2}+\alpha_2-2\beta_2.
%\end{eqnarray*}
%Such methods are A-stable for $\alpha_2\geq \frac{1}{2}$ and $\beta\geq \frac{\alpha_2}{2}$ and strongly A-stable for $\alpha_2> \frac{1}{2}$ and $\beta> \frac{\alpha_2}{2}$.
\begin{lemma}
\cite{VGAP}
\label{lemma}
Let $\delta =\beta_2-\frac{\alpha_2}{2}>0$. Then the coefficients $\alpha_i$ and $\beta_i$ satisfy the following relation:
\begin{eqnarray*}
&& 2\left(\sum_{i=0}^{2}\alpha_i\vartheta_i\right)\left(\sum_{i=0}^{2}\beta_i\vartheta_i\right)
\geq (\alpha_2^2+\delta )\vartheta_2^2-(2\alpha_2-1)\vartheta_1^2-\left((\alpha_2-1)^2+\delta\right)\vartheta_0^2\\
&&\qquad\qquad\qquad\qquad-2\left(\alpha_2(\alpha_2-1)+\delta\right)
(\vartheta_2\vartheta_1-\vartheta_1\vartheta_0),\quad
\forall\ \vartheta_0, \vartheta_1, \vartheta_2\in R.
\end{eqnarray*}
\end{lemma}

\begin{lemma}\cite{MMXZ} \label{gamma}For
$\forall$ $\vec{v}_f\in X_{f}$, $\phi_p\in X_{p}$, there exists $C_k>0$ such that $\forall$ $\varepsilon>0$,
\begin{eqnarray}\label{gammal}
\vert c_\Gamma(\vec{v}_f,\phi_p)\vert
\leq\frac{1}{4\varepsilon}\Vert \vec{v}_f \Vert_{X_{f}}^2
+C_k\varepsilon\Vert\phi_p\Vert_{{X}_p}^2.
\end{eqnarray}
In addition, for $\forall$ $\vec{v}_{f}^{h}\in {X}_{fh}$, $\phi_{p}^{h}\in {X}_{ph}$, there exists $\tilde{C}_k>0$ such that $\forall$ $\tilde{\varepsilon}>0$,
\begin{eqnarray}\label{gamma2}
\vert c_\Gamma(\vec{v}_{f}^{h},\phi_{p}^{h})\vert
\leq\frac{1}{4\tilde{\varepsilon}}\Vert\vec{v}_{f}^{h}\Vert_{{X}_f}^2
+\tilde{C}_k\tilde{\varepsilon} h^{-1}\|\phi_{p}^{h}\|_{p}^2.
\end{eqnarray}
\end{lemma}

\begin{lemma}\cite{AFDS} \label{gronwall}(Discrete Gronwall Inequality)
Let $\Delta t$, $C$,  $a_i$, $b_i$, $c_i$, $d_i$, (for integers $n\geq0$) be non-negative numbers such that
\begin{eqnarray}\label{gronwall1}
a_n+\Delta t \sum_{i=0}^{n}b_i\leq+\Delta t \sum_{i=0}^{n-1}c_i+C,\ \forall \ n\geq1,
\end{eqnarray}
 then
\begin{eqnarray}\label{gronwall2}
a_n+\Delta t \sum_{i=0}^{n}b_i\leq
exp\left(\Delta t \sum_{i=0}^{n-1}d_i\right)
    \left(\Delta t \sum_{i=0}^{n-1}c_i+C\right),\ \forall \ n\geq1.
\end{eqnarray}
\end{lemma}
%\begin{lemma}\cite{EAST,ATAD} \label{gronwall}(Discrete Gronwall Inequality)
%Let $\Delta t$, $C$,  $a_i$, $b_i$, $c_i$, $d_i$, (for integers $n\geq0$) be non-negative numbers such that
%\begin{eqnarray}\label{gronwall1}
%a_n+\Delta t \sum_{i=0}^{n}b_i\leq+\Delta t \sum_{i=0}^{n}c_i+C,\ \forall \ n\geq0,
%\end{eqnarray}
%Suppose $\Delta t d_i<1$ $\forall \ n$, then
%\begin{eqnarray}\label{gronwall2}
%a_n+\Delta t \sum_{i=0}^{n}b_i\leq
%exp\left(\Delta t \sum_{i=0}^{n}b_i\frac{d_i}{1-\Delta t d_i}\right)
%    \left(\Delta t \sum_{i=0}^{n}c_i+C\right),\ \forall \ n\geq0,
%\end{eqnarray}
%\end{lemma}
For the rest of the paper, $P=\{t_m\}_{m=0}^N$ is the partition on time interval $t_0=0$, $t_N=T$, $k_m=t_{m+1}-t_m$ is the time step size, and $\tau_m=\frac{k_{m+1}}{k_m}$ is a ratio for the time step and satisfies $\tau_{min}\leq\tau_m\leq\tau_{max}$.
Here $(\vec{\ul}^{h,{m+1}},p_{f}^{h,m+1})=(\vec{u}_{f}^{h,m+1},p_{f}^{h,m+1},\phi_{p}^{h,m+1})$  denote the approximation solutions by $(\vec{\ul}^{h}(t_{m+1}), p_{f}^{h}(t_{m+1}))=(\vec{u}_{f}^{h}(t_{m+1}),p_{f}^{h}(t_{m+1}),\phi_{p}^{h}(t_{m+1}))$.
\subsection{Numberical algorithms}
First, we introduce the coupled and decoupled variable time-stepping algorithms.

\textbf{Algorithm 1.(Coupled variable time-stepping algorithm for Linear Multi-step method plus time filter)}

$\blacktriangledown$
The Linear Multi-step method (First Order):

Give $({\vec{\ul}}^{h,0},\ {\vec{p}}_f^{h,0})$ and $({\vec{\ul}}^{h,1},\ {\vec{p}}_f^{h,1})$,
find $\hat{\vec{\ul}}^{h,m+1}=(\hat{\vec{u}}_{f}^{h,m+1}$, $\hat{\phi}_{p}^{h,m+1})\in X_h$ and $\hat{p}_{f}^{h,m+1}\in Q_{fh}$ with $m=1, 2, ... , N-1$, such that $\forall$ $\vl^{h}\in X_h$ and $q_{f}^{h}\in Q_{fh}$,
\begin{eqnarray}\label{clm}
&&\left(\frac{\hat{\vec{\ul}}^{h,m+1}-\vec{\ul}^{h,m}}{k_m},\vec{\vl}^{h}\right)
+a((1-\theta)\hat{\vec{\ul}}^{h,m+1}+\theta\vec{\ul}^{h,m},\vec{\vl}^{h})
+b(\vec{\vl}^{h},(1-\theta)\hat{p}_{f}^{h,m+1}+\theta p_{f}^{h,m})\nonumber\\
&&=\big<(1-\theta)\vec{F}^{m+1}+\theta\vec{F}^{m},\vec{\vl}^{h}\big>,\\
&&b((1-\theta)\hat{\vec{\ul}}^{h,m+1}+\theta\vec{\ul}^{h,m},q_{f}^{h})=0.\nonumber
\end{eqnarray}

$\blacktriangledown$
The Time Filter (Second Order):

Update the previous solutions $(\hat{\vec{\ul}}^{h,m+1},\hat{p}_{f}^{h,m+1})$ by time filter,
\begin{eqnarray}
\begin{split}\label{ctf}
&\vec{\ul}^{h,m+1}=\hat{\vec{\ul}}^{h,m+1}-\frac{(1-2\theta)(1+\tau_{m-1})\tau_{m-1}}{2(1-\theta)\tau_{m-1}+1}
\left(\frac{1}{1+\tau_{m-1}}\hat{\vec{\ul}}^{h,m+1}-\vec{\ul}^{h,m}+\frac{\tau_{m-1}}{1+\tau_{m-1}}\vec{\ul}^{h,m-1}\right),\\
&p_{f}^{h,m+1}=\hat{p}_{f}^{h,m+1}-\frac{(1-2\theta)(1+\tau_{m-1})\tau_{m-1}}{2(1-\theta)\tau_{m-1}+1}
\left(\frac{1}{1+\tau_{m-1}}\hat{p}_{f}^{h,m+1}-p_{f}^{h,m}+\frac{\tau_{m-1}}{1+\tau_{m-1}}p_{f}^{h,m-1}\right).
\end{split}
\end{eqnarray}

\textbf{Algorithm 2.(decoupled variable time-stepping algorithm for Linear Multi-step method plus time filter)}

$\blacktriangledown$
The Linear Multi-step method (First Order):

Give $({\vec{u}}_{f}^{h,0},\ {p}_{f}^{h,0})$ and  $({\vec{u}}_{f}^{h,1},\ {p}_{f}^{h,1})$,
find $ (\hat{\vec{u}}_f^{h,m+1},\ \hat{p}_{f}^{h,m+1})\in(X_{fh},Q_{fh})$ with $m=1,2,......,N-1$, such that for $\forall$ $\vec{v}_{f}^{h} \in X _{fh}$ and $q_{f}^{h}\in Q_{fh}$,
\begin{eqnarray}\label{ds}
&&\left(\frac{\hat{\vec{u}}_f^{h,m+1}-\vec{u}_{f}^{h,m}}{k_m},\vec{v}_{f}^{h}\right)_{\Omega_f}
+a_{\Omega_f}((1-\theta)\hat{\vec{u}}_f^{h,m+1}+\theta\vec{u}_{f}^{h,m},\vec{v}_{f}^{h})
+b(\vec{v}_{f}^{h},(1-\theta)\hat{p}_{f}^{h,m+1}+\theta p_{f}^{h,m})\nonumber\\
&&=((1-\theta)\vec{g}_{f}^{m+1}+\theta\vec{g}_{f}^{m},\vec{v}_{f}^{h})_{\Omega_f}
-c_{\Gamma}(\vec{v}_{f}^{h},(1+(1-\theta)\tau_{m-1})\phi_{p}^{h,m}
-(1-\theta)\tau_{m-1}\phi_{p}^{h,m-1}),\\
&& b((1-\theta)\hat{\vec{u}}_f^{h,m+1}+\theta\vec{u}_{f}^{h,m},q_{f}^{h})=0.\nonumber
\end{eqnarray}

Give ${\phi}_{p}^{h,0}$ and ${\phi}_{p}^{h,1}$, find $\hat{\phi}_{p}^{h,m+1}\in X_{ph}$ with $m=1,2,......,N-1$, such that for $\forall$ $\psi_{p}^{h} \in X_{ph}$,
\begin{eqnarray}
\begin{split}\label{dd}
&g\left(\frac{\hat{\phi}_{p}^{h,m+1}-\phi_{p}^{h,m}}{k_m},\psi_{p}^{h}\right)_{{\Omega}_p}
+a_{\Omega_p}((1-\theta)\hat{\phi}_{p}^{h,m+1}+\theta\phi_{p}^{h,m},\psi_{p}^{h})\\
&=g((1-\theta)g_{p}^{m+1}+\theta g_{p}^{m},\psi_{p}^{h})_{\Omega_p}
+c_{\Gamma}((1+(1-\theta)\tau_{m-1})\vec{u}_{f}^{h,m}
-(1-\theta)\tau_{m-1}\vec{u}_{f}^{h,m-1},\psi_{p}^{h}).
\end{split}
\end{eqnarray}

$\blacktriangledown$The Time Filter (Second Order):

Update the previous solutions $(\hat{\vec{u}}_{f}^{h,m+1},\hat{p}_{f}^{h,m+1},\hat{\phi}_{p}^{h,m+1})$ by time filter,
\begin{eqnarray}\label{dtf}
\begin{split}
&\vec{u}_{f}^{h,m+1}=\hat{\vec{u}}_{f}^{h,m+1}-\frac{(1-2\theta)(1+\tau_{m-1})\tau_{m-1}}{2(1-\theta)\tau_{m-1}+1}
\left(\frac{1}{1+\tau_{m-1}}\hat{\vec{u}}_{f}^{h,m+1}-\vec{u}_{f}^{h,m}+\frac{\tau_{m-1}}{1+\tau_{m-1}}\vec{u}_{f}^{h,m-1}\right),\\
&p_{f}^{h,m+1}=\hat{p}_{f}^{h,m+1}-\frac{(1-2\theta)(1+\tau_{m-1})\tau_{m-1}}{2(1-\theta)\tau_{m-1}+1}
\left(\frac{1}{1+\tau_{m-1}}\hat{p}_{f}^{h,m+1}-p_{f}^{h,m}+\frac{\tau_{m-1}}{1+\tau_{m-1}}p_{f}^{h,m-1}\right),\\
&\phi_{p}^{h,m+1}=\hat{\phi}_{p}^{h,m+1}-\frac{(1-2\theta)(1+\tau_{m-1})\tau_{m-1}}{2(1-\theta)\tau_{m-1}+1}
\left(\frac{1}{1+\tau_{m-1}}\hat{\phi}_{p}^{h,m+1}-\phi_{p}^{h,m}+\frac{\tau_{m-1}}{1+\tau_{m-1}}\phi_{p}^{h,m-1}\right).
\end{split}
\end{eqnarray}

Then, we introduce coupled and decoupled adaptive algorithms.
Here a combination of some common adaptive method is used to select the time step. $\varepsilon$ denote tolerance,   $\hat{\gamma}$ and $\check{\gamma}$ represent two safety factors, respectively. The role of the first safety factor is to prevent the size of the next step from being too large, thereby reducing the possibility that the next solution will be rejected. The second safety factor is to make the time step increase more slowly, so that the possibility of the recalculated result being accepted increases, and proceeds to the next calculation.
In the numerical experiment of this paper, we choose $\hat{\gamma}=0.9$ and $\check{\gamma}=0.6$.
And using the definition in \cite{APMF}, the $j$th order divided difference describe as $\delta^j y=y[t_{n+m}, t_{n+m-1},...,t_{n+m-j}]$ and the parameter in time filter describe as $\chi^{l+1}=\frac{\prod^{l}_{i=1}(t_{n+m}-t_{n+m-i})}{\prod^{l+1}_{j=1}(t_{n+m}-t_{n+m-j})^{-1}}$.

\textbf{Algorithm 3.(coupled adaptive algorithm for Linear Multi-step method plus time filter)}

Let $m=2$  and give $\varepsilon$, $\hat{\gamma}$, $\check{\gamma}$, $({\vec{\ul}}^{h,m-2},\ {\vec{p}}_f^{h,m-2})$, $({\vec{\ul}}^{h,m-1},\ {\vec{p}}_f^{h,m-1})$ and $({\vec{\ul}}^{h,m},\ {\vec{p}}_f^{h,m})$,
 compute ${\vec{\ul}}^{h,m+1}=({\vec{u}}_{f}^{h,m+1},\ {\phi}_{p}^{h,m+1})$ and ${p}_{f}^{h,m+1}$ by solving

$\blacktriangledown$
The Linear Multi-step method (First Order):
\begin{eqnarray}\label{clm}
&&\left(\frac{\hat{\vec{\ul}}^{h,m+1}-\vec{\ul}^{h,m}}{k_m},\vec{\vl}^{h}\right)
+a((1-\theta)\hat{\vec{\ul}}^{h,m+1}+\theta\vec{\ul}^{h,m},\vec{\vl}^{h})
+b(\vec{\vl}^{h},(1-\theta)\hat{p}_{f}^{h,m+1}+\theta p_{f}^{h,m})\nonumber\\
&&=\big<(1-\theta)\vec{F}^{m+1}+\theta\vec{F}^{m},\vec{\vl}^{h}\big>,\\
&&b((1-\theta)\hat{\vec{\ul}}^{h,m+1}+\theta\vec{\ul}^{h,m},q_{f}^{h})=0.\nonumber
\end{eqnarray}

$\blacktriangledown$
The Time Filter (Second Order):

Update the previous solutions $(\hat{\vec{\ul}}^{h,m+1},\hat{p}_{f}^{h,m+1})$ by time filter,
\begin{eqnarray}
\begin{split}\label{ctf}
&\vec{\ul}^{h,m+1}=\hat{\vec{\ul}}^{h,m+1}-\frac{(1-2\theta)(1+\tau_{m-1})\tau_{m-1}}{2(1-\theta)\tau_{m-1}+1}
\left(\frac{1}{1+\tau_{m-1}}\hat{\vec{\ul}}^{h,m+1}-\vec{\ul}^{h,m}+\frac{\tau_{m-1}}{1+\tau_{m-1}}\vec{\ul}^{h,m-1}\right),\\
&p_{f}^{h,m+1}=\hat{p}_{f}^{h,m+1}-\frac{(1-2\theta)(1+\tau_{m-1})\tau_{m-1}}{2(1-\theta)\tau_{m-1}+1}
\left(\frac{1}{1+\tau_{m-1}}\hat{p}_{f}^{h,m+1}-p_{f}^{h,m}+\frac{\tau_{m-1}}{1+\tau_{m-1}}p_{f}^{h,m-1}\right).
\end{split}
\end{eqnarray}

Choose
\begin{eqnarray*}
EST_{\ul}=\chi^{3}\delta^{3}\vec{\ul}^{h,m+1},
\end{eqnarray*}
that is
\begin{eqnarray*}
EST_{\vec{\u}_f}=\chi^{3}\delta^{3}\vec{\u}_{f}^{h,m+1},\quad EST_{\phi_p}=\chi^{3}\delta^{3}{\phi}_p^{h,m+1}.
\end{eqnarray*}
If $min\{|EST_{\vec{\u}_f}|,|EST_{\phi_p}|\}<\frac{\varepsilon}{4}$,
\begin{eqnarray*}
\sigma_m=min\left\{2,\left(\frac{\varepsilon}{|EST_{\vec{\u}_f}|}\right)^{\frac{1}{3}},
\left(\frac{\varepsilon}{|EST_{\phi_p}|}\right)^{\frac{1}{3}}\right\},\qquad
k_{m+1}=\hat{\gamma}\cdot\sigma_m\cdot k_m,
\end{eqnarray*}
if $\frac{\varepsilon}{4}\leq min\{|EST_{\vec{\u}_f}|,|EST_{\phi_p}|\}\leq\varepsilon$,
\begin{eqnarray*}
\sigma_m=min\left\{1,\left(\frac{\varepsilon}{|EST_{\vec{\u}_f}|}\right)^{\frac{1}{3}},
\left(\frac{\varepsilon}{|EST_{\phi_p}|}\right)^{\frac{1}{3}}\right\},\qquad
k_{m+1}=\hat{\gamma}\cdot\sigma_m\cdot k_{m}.
\end{eqnarray*}
If the above situations are not statify, let
\begin{eqnarray*}
k_{m}=\check{\gamma}\cdot\sigma_{m-1}\cdot k_{m-1}
\end{eqnarray*}
and recompute the above steps.

\textbf{Algorithm 4.(decoupled variable time-stepping algorithm for Linear Multi-step method plus time filter)}

Let m=2 and give $\varepsilon$, $\hat{\gamma}$, $\check{\gamma}$, $({\vec{u}}_{f}^{h,m-2},\ \hat{p}_{f}^{h,m-2},\ \phi_{p}^{h,m-2})$, $({\vec{u}}_{f}^{h,m-1},\ {p}_{f}^{h,m-1},\ \phi_{p}^{h,m-1})$ and $({\vec{u}}_{f}^{h,m},\ {p}_{f}^{h,m},\ \phi_{p}^{h,m})$,
 compute $({\vec{u}}_f^{h,m+1},\ \hat{p}_{f}^{h,m+1},\ \phi_{p}^{m+1})$ by solving

$\blacktriangledown$
The Linear Multi-step method (First Order):
\begin{eqnarray}\label{ds}
&&\left(\frac{\hat{\vec{u}}_f^{h,m+1}-\vec{u}_{f}^{h,m}}{k_m},\vec{v}_{f}^{h}\right)_{\Omega_f}
+a_{\Omega_f}((1-\theta)\hat{\vec{u}}_f^{h,m+1}+\theta\vec{u}_{f}^{h,m},\vec{v}_{f}^{h})
+b(\vec{v}_{f}^{h},(1-\theta)\hat{p}_{f}^{h,m+1}+\theta p_{f}^{h,m})\nonumber\\
&&=((1-\theta)\vec{g}_{f}^{m+1}+\theta\vec{g}_{f}^{m},\vec{v}_{f}^{h})_{\Omega_f}
-c_{\Gamma}(\vec{v}_{f}^{h},(1+(1-\theta)\tau_{m-1})\phi_{p}^{h,m}
-(1-\theta)\tau_{m-1}\phi_{p}^{h,m-1}),\\
&& b((1-\theta)\hat{\vec{u}}_f^{h,m+1}+\theta\vec{u}_{f}^{h,m},q_{f}^{h})=0.\nonumber
\end{eqnarray}
\begin{eqnarray}
\begin{split}\label{dd}
&g\left(\frac{\hat{\phi}_{p}^{h,m+1}-\phi_{p}^{h,m}}{k_m},\psi_{p}^{h}\right)_{{\Omega}_p}
+a_{\Omega_p}((1-\theta)\hat{\phi}_{p}^{h,m+1}+\theta\phi_{p}^{h,m},\psi_{p}^{h})\\
&=g((1-\theta)g_{p}^{m+1}+\theta g_{p}^{m},\psi_{p}^{h})_{\Omega_p}
+c_{\Gamma}((1+(1-\theta)\tau_{m-1})\vec{u}_{f}^{h,m}
-(1-\theta)\tau_{m-1}\vec{u}_{f}^{h,m-1},\psi_{p}^{h}).
\end{split}
\end{eqnarray}

$\blacktriangledown$The Time Filter (Second Order):

Update the previous solutions $(\hat{\vec{u}}_{f}^{h,m+1},\ {p}_{f}^{h,m+1},{\phi}_{p}^{h,m+1})$ by time filter,
\begin{eqnarray}\label{dtf}
\begin{split}
&\vec{u}_{f}^{h,m+1}=\hat{\vec{u}}_{f}^{h,m+1}-\frac{(1-2\theta)(1+\tau_{m-1})\tau_{m-1}}{2(1-\theta)\tau_{m-1}+1}
\left(\frac{1}{1+\tau_{m-1}}\hat{\vec{u}}_{f}^{h,m+1}-\vec{u}_{f}^{h,m}+\frac{\tau_{m-1}}{1+\tau_{m-1}}\vec{u}_{f}^{h,m-1}\right),\\
&p_{f}^{h,m+1}=\hat{p}_{f}^{h,m+1}-\frac{(1-2\theta)(1+\tau_{m-1})\tau_{m-1}}{2(1-\theta)\tau_{m-1}+1}
\left(\frac{1}{1+\tau_{m-1}}\hat{p}_{f}^{h,m+1}-p_{f}^{h,m}+\frac{\tau_{m-1}}{1+\tau_{m-1}}p_{f}^{h,m-1}\right),\\
&\phi_{p}^{h,m+1}=\hat{\phi}_{p}^{h,m+1}-\frac{(1-2\theta)(1+\tau_{m-1})\tau_{m-1}}{2(1-\theta)\tau_{m-1}+1}
\left(\frac{1}{1+\tau_{m-1}}\hat{\phi}_{p}^{h,m+1}-\phi_{p}^{h,m}+\frac{\tau_{m-1}}{1+\tau_{m-1}}\phi_{p}^{h,m-1}\right).
\end{split}
\end{eqnarray}

Choose
\begin{eqnarray*}
EST_{\vec{\u}_f}=\chi^{3}\delta^{3}\vec{\u}_{f}^{h,m+1},\quad EST_{\phi_p}=\chi^{3}\delta^{3}{\phi}_p^{h,m+1}.
\end{eqnarray*}
If $min\{|EST_{\vec{\u}_f}|,|EST_{\phi_p}|\}<\frac{\varepsilon}{4}$,
\begin{eqnarray*}
\sigma_m=min\left\{2,\left(\frac{\varepsilon}{|EST_{\vec{\u}_f}|}\right)^{\frac{1}{3}},
\left(\frac{\varepsilon}{|EST_{\phi_p}|}\right)^{\frac{1}{3}}\right\},\qquad
k_{m+1}=\hat{\gamma}\cdot\sigma_m\cdot k_m,
\end{eqnarray*}
if $\frac{\varepsilon}{4}\leq min\{|EST_{\vec{\u}_f}|,|EST_{\phi_p}|\}\leq\varepsilon$,
\begin{eqnarray*}
\sigma_m=min\left\{1,\left(\frac{\varepsilon}{|EST_{\vec{\u}_f}|}\right)^{\frac{1}{3}},
\left(\frac{\varepsilon}{|EST_{\phi_p}|}\right)^{\frac{1}{3}}\right\},\qquad
k_{m+1}=\hat{\gamma}\cdot\sigma_m\cdot k_{m}.
\end{eqnarray*}
If the above situations are not statify, let
\begin{eqnarray*}
k_{m}=\check{\gamma}\cdot\sigma_{m-1}\cdot k_{m-1}
\end{eqnarray*}
and recompute the above steps.

\textbf{Remark 1:} In the decoupled algorithm, we use the second-order extrapolation method to approximate $\hat{\vec{u}}_{f}^{h,m+1}$ with $(1+\tau_{m-1})\vec{u}_{f}^{h,m}-\tau_{m-1}\vec{u}_{f}^{h,m-1}$ and $\hat{\phi}_{p}^{h,m+1}$ with $(1+\tau_{m-1})\phi_{p}^{h,m}-\tau_{m-1}\phi_{p}^{h,m-1}$ in the interface coupled term. At the same time, whether the pressure $\hat{p}_{f}^{h,m+1}$  is filtered or not has little influence on the result.

For the convenience of comparison in the later experiment, we present the coupled and decoupled adaptive algorithms for the Linear Multi-step method.

\textbf{Algorithm 5.(Coupled adaptive algorithm for Linear Multi-step method)}

Let $m=1$ and give $\varepsilon$, $\hat{\gamma}$, $\check{\gamma}$,
$({\vec{\ul}}^{h,m-1},\ {\vec{p}}_f^{h,m-1})$ and $({\vec{\ul}}^{h,m},\ {\vec{p}}_f^{h,m})$,
compute ${\vec{\ul}}^{h,m+1}=(\hat{\vec{u}}_{f}^{h,m+1}$, $\hat{\phi}_{p}^{h,m+1})$ and ${p}_{f}^{h,m+1}$, by solving

$\blacktriangledown$
The Linear Multi-step method (First Order):
\begin{eqnarray}\label{clm}
&&\left(\frac{{\vec{\ul}}^{h,m+1}-\vec{\ul}^{h,m}}{k_m},\vec{\vl}^{h}\right)
+a((1-\theta){\vec{\ul}}^{h,m+1}+\theta\vec{\ul}^{h,m},\vec{\vl}^{h})
+b(\vec{\vl}^{h},(1-\theta){p}_{f}^{h,m+1}+\theta p_{f}^{h,m})\nonumber\\
&&=\big<(1-\theta)\vec{F}^{m+1}+\theta\vec{F}^{m},\vec{\vl}^{h}\big>,\\
&&b((1-\theta){\vec{\ul}}^{h,m+1}+\theta\vec{\ul}^{h,m},q_{f}^{h})=0.\nonumber
\end{eqnarray}

Choose
\begin{eqnarray*}
EST_{\ul}=\frac{(1-2\theta)(1+\tau_{m-1})\tau_{m-1}}{2(1-\theta)\tau_{m-1}+1}
\left(\frac{1}{1+\tau_{m-1}}{\vec{\ul}}^{h,m+1}-\vec{\ul}^{h,m}+\frac{\tau_{m-1}}{1+\tau_{m-1}}\vec{\ul}^{h,m-1}\right),
\end{eqnarray*}
that is
\begin{eqnarray*}
&EST_{\vec{\u}_f}=\frac{(1-2\theta)(1+\tau_{m-1})\tau_{m-1}}{2(1-\theta)\tau_{m-1}+1}
\left(\frac{1}{1+\tau_{m-1}}{\vec{\u}_f}^{h,m+1}-\vec{\u}_f^{h,m}+\frac{\tau_{m-1}}{1+\tau_{m-1}}\vec{\u}_f^{h,m-1}\right),\\
&EST_{\phi_p}=\frac{(1-2\theta)(1+\tau_{m-1})\tau_{m-1}}{2(1-\theta)\tau_{m-1}+1}
\left(\frac{1}{1+\tau_{m-1}}{\phi}_p^{h,m+1}-{\phi}_p^{h,m}+\frac{\tau_{m-1}}{1+\tau_{m-1}}{\phi}_p^{h,m-1}\right).
\end{eqnarray*}
If $min\{|EST_{\vec{\u}_f}|,|EST_{\phi_p}|\}<\frac{\varepsilon}{4}$,
\begin{eqnarray*}
\sigma_m=min\left\{2,\left(\frac{\varepsilon}{|EST_{\vec{\u}_f}|}\right)^{\frac{1}{2}},
\left(\frac{\varepsilon}{|EST_{\phi_p}|}\right)^{\frac{1}{2}}\right\},\qquad
k_{m+1}=\hat{\gamma}\cdot\sigma_m\cdot k_{m},
\end{eqnarray*}
if $\frac{\varepsilon}{4}\leq min\{|EST_{\vec{\u}_f}|,|EST_{\phi_p}|\}\leq\varepsilon$,
\begin{eqnarray*}
\sigma_m=min\left\{1,\left(\frac{\varepsilon}{|EST_{\vec{\u}_f}|}\right)^{\frac{1}{2}},
\left(\frac{\varepsilon}{|EST_{\phi_p}|}\right)^{\frac{1}{2}}\right\},\qquad
k_{m+1}=\hat{\gamma}\cdot\sigma_m\cdot k_{m}.
\end{eqnarray*}
If the above situations are not statify, let
\begin{eqnarray*}
k_{m}=\check{\gamma}\cdot\sigma_{m-1}\cdot k_{m-1}
\end{eqnarray*}
and recompute the above steps.

\textbf{Algorithm 6.(decoupled variable time-stepping algorithm for Linear Multi-step method)}

Let $m=1$ and give $\varepsilon$, $\hat{\gamma}$, $\check{\gamma}$, $({\vec{u}}_{f}^{h,m-1},\ {p}_{f}^{h,m-1},\ \phi_{p}^{h,m-1})$ and  $({\vec{u}}_{f}^{h,m},\ {p}_{f}^{h,m},\ \phi_{p}^{h,m})$,
 compute $({\vec{u}}_f^{h,m+1},\ {p}_{f}^{h,m+1},\ {\phi_{p}^{h,m+1}})$ by solving

$\blacktriangledown$
The Linear Multi-step method (First Order):
\begin{eqnarray}
&&\left(\frac{{\vec{u}}_f^{h,m+1}-\vec{u}_{f}^{h,m}}{k_m},\vec{v}_{f}^{h}\right)_{\Omega_f}
+a_{\Omega_f}((1-\theta){\vec{u}}_f^{h,m+1}+\theta\vec{u}_{f}^{h,m},\vec{v}_{f}^{h})
+b(\vec{v}_{f}^{h},(1-\theta){p}_{f}^{h,m+1}+\theta p_{f}^{h,m})\nonumber\\
&&=((1-\theta)\vec{g}_{f}^{m+1}+\theta\vec{g}_{f}^{m},\vec{v}_{f}^{h})_{\Omega_f}
-c_{\Gamma}(\vec{v}_{f}^{h},\phi_{p}^{h,m}),\\
&& b((1-\theta){\vec{u}}_f^{h,m+1}+\theta\vec{u}_{f}^{h,m},q_{f}^{h})=0.\nonumber
\end{eqnarray}
\begin{eqnarray}
\begin{split}
&g\left(\frac{\hat{\phi}_{p}^{h,m+1}-\phi_{p}^{h,m}}{k_m},\psi_{p}^{h}\right)_{{\Omega}_p}
+a_{\Omega_p}((1-\theta)\hat{\phi}_{p}^{h,m+1}+\theta\phi_{p}^{h,m},\psi_{p}^{h})\\
&=g((1-\theta)g_{p}^{m+1}+\theta g_{p}^{m},\psi_{p}^{h})_{\Omega_p}
+c_{\Gamma}(\vec{u}_{f}^{h,m},\psi_{p}^{h}).
\end{split}
\end{eqnarray}

Choose
\begin{eqnarray*}
&EST_{\vec{\u}_f}=\frac{(1-2\theta)(1+\tau_{m-1})\tau_{m-1}}{2(1-\theta)\tau_{m-1}+1}
\left(\frac{1}{1+\tau_{m-1}}{\vec{\u}_f}^{h,m+1}-\vec{\u}_f^{h,m}+\frac{\tau_{m-1}}{1+\tau_{m-1}}\vec{\u}_f^{h,m-1}\right),\\
&EST_{\phi_p}=\frac{(1-2\theta)(1+\tau_{m-1})\tau_{m-1}}{2(1-\theta)\tau_{m-1}+1}
\left(\frac{1}{1+\tau_{m-1}}{\phi}_p^{h,m+1}-{\phi}_p^{h,m}+\frac{\tau_{m-1}}{1+\tau_{m-1}}{\phi}_p^{h,m-1}\right).
\end{eqnarray*}
If $min\{|EST_{\vec{\u}_f}|,|EST_{\phi_p}|\}<\frac{\varepsilon}{4}$,
\begin{eqnarray*}
\sigma_m=min\left\{2,\left(\frac{\varepsilon}{|EST_{\vec{\u}_f}|}\right)^{\frac{1}{2}},
\left(\frac{\varepsilon}{|EST_{\phi_p}|}\right)^{\frac{1}{2}}\right\},\qquad
k_{m+1}=\hat{\gamma}\cdot\sigma_m\cdot k_{m},
\end{eqnarray*}
if $\frac{\varepsilon}{4}\leq min\{|EST_{\vec{\u}_f}|,|EST_{\phi_p}|\}\leq\varepsilon$,
\begin{eqnarray*}
\sigma_m=min\left\{1,\left(\frac{\varepsilon}{|EST_{\vec{\u}_f}|}\right)^{\frac{1}{2}},
\left(\frac{\varepsilon}{|EST_{\phi_p}|}\right)^{\frac{1}{2}}\right\},\qquad
k_{m+1}=\hat{\gamma}\cdot\sigma_m\cdot k_{m}.
\end{eqnarray*}
If the above situations are not statify, let
\begin{eqnarray*}
k_{m}=\check{\gamma}\cdot\sigma_{m-1}\cdot k_{m-1}
\end{eqnarray*}
and recompute the above steps.

\textbf{Remark 2:} In the decoupled algorithm for the Linear Multi-step method, we use the first-order extrapolation method to approximate $\hat{\vec{u}}_{f}^{h,m+1}$ with $\vec{u}_{f}^{h,m}$ and $\hat{\phi}_{p}^{h,m+1}$ with $\phi_{p}^{h,m}$ in the interface coupled term. Similarly, whether the pressure $\hat{p}_{f}^{h,m+1}$  is filtered or not has little influence on the result.

We define some notations for the following analysis:
\begin{align*}
&A(\vec{\ul}^{h,m+1})=\frac{2(1-\theta)\tau_{m-1}+1}{\tau_{m-1}+1}\vec{\ul}^{h,m+1}
-((1-2\theta)\tau_{m-1}+1)\vec{\ul}^{h,m}
+\frac{(1-2\theta)\tau_{m-1}^2}{\tau_{m-1}+1}\vec{\ul}^{h,m-1},\\
&B(\vec{\ul}^{h,m+1})=\frac{2(1-\theta)^2\tau_{m-1}+1-\theta}{\tau_{m-1}+1}\vec{\ul}^{h,m+1}
         -((1-\theta)(1-2\theta)\tau_{m-1}-\theta)\vec{\ul}^{h,m}
         +\frac{(1-\theta)(1-2\theta)\tau_{m-1}^2}{\tau_{m-1}+1}\vec{\ul}^{h,m-1},\\
&S(\vec{F}^{m+1})=-\frac{(1-\theta)(1-2\theta)\tau_{m-1}}{\tau_{m-1}+1}\vec{F}^{m+1}
   +((1-\theta)(1-2\theta)\tau_{m-1}-\theta)\vec{F}^{m}
   -\frac{(1-\theta)(1-2\theta)\tau_{m-1}^2}{\tau_{m-1}+1}\vec{F}^{m-1},\\
&W(\eta_{\phi_p}^{h,m+1})=(2(1-\theta)^2\tau_{m-1}+1-\theta)\eta_{\phi_p}^{h,m}
-\frac{2(1-\theta)^2\tau_{m-1}^2+(1-\theta)\tau_{m-1}}{\tau_{m-1}+1}\eta_{\phi_p}^{h,m-1}
-\frac{2(1-\theta)^2\tau_{m-1}+1-\theta}{\tau_{m-1}+1}\eta_{\phi_p}^{h,m+1},\\
&D(\theta,\tau_{m-1})
=\frac{2(1-\theta)(5-6\theta)\tau_{m-1}^2+(4\theta^2-16\theta+11)\tau_{m-1}+3-2\theta}
{2(\tau_{m-1}+1)^2}>0,\\
&H(\theta,\tau_{m-1})
=\frac{2(3-4\theta)\tau_{m-1}^2+8(1-\theta)\tau_{m-1}+2}
{2(\tau_{m-1}+1)^2}>0,\\
&E(\theta,\tau_{m-1})
=\frac{2(1-2\theta)(2-3\theta)\tau_{m-1}^2+(1-2\theta)(3-2\theta)\tau_{m-1}+1-2\theta}
{2(\tau_{m-1}+1)^2}>0,\\
&F(\theta,\tau_{m-1})
=\frac{12(1-\theta)(1-2\theta)\tau_{m-1}^2+2(1-2\theta)(5-2\theta)\tau_{m-1}+2(1-2\theta)}
{2(\tau_{m-1}+1)^2}>0,\\
&G(\theta,\tau_{m-1})
=\frac{6(1-\theta)(1-2\theta)\tau_{m-1}^2+(1-2\theta)(5-2\theta)\tau_{m-1}+1-2\theta}
      {2(\tau_{m-1}+1)^2}
 \cdot
 \frac{6(1-\theta)\tau_{m-1}^2+(5-2\theta)\tau_{m-1}+1}
      {2(2-3\theta)\tau_{m-1}^2+(3-2\theta)\tau_{m-1}+1}\\
&\qquad\qquad\ \ =\frac{(3\tau_{m-1}+1)^2(2\tau_{m-1}\theta-2\tau_{m-1}-1)^2(-1+2\theta)}
{2(\tau_{m-1}+1)^2(6\tau_{m-1}^2\theta-4\tau_{m-1}^2+2\tau_{m-1}\theta-3\tau_{m-1}-1)}>0,\\
&I(\theta,\tau_{m-1})=D(\theta,\tau_{m-1})-G(\theta,\tau_{m-1})
=\frac{(\tau_{m-1}+1)(2\tau_{m-1}\theta-2\tau_{m-1}-1)}
{6\tau_{m-1}^2\theta-4\tau_{m-1}^2+2\tau_{m-1}\theta-3\tau_{m-1}-1}>0.
\end{align*}
\subsection{Stabilities analysis}
First, we give following stability theorem of coupled variable time-stepping algorithm(\textbf{Algorthm 1}).
\begin{theorem}\label{theorem1} (Stability of \textbf{Algorthm 1}) Let $\vec{\ul}^{h,m+1}$ be the solution of the Linear Multi-step methods plus time filter with $0<\theta<\frac{1}{2}$. For $N\geq 2$, we have
\begin{align*}
&I(\theta,\tau)\|\vec{\ul}^{h,N}\|^2_0
+C_{coe}\sum_{m=1}^{N-1}\big[k_m\|B(\vec{\ul}^{h,m+1})\|_X^2\big]\\
\leq& C \big(\|\vec{g}_f\|^2_{L^2(0,T;L^2(\Omega_f))}+\|g_{p}\|^2_{L^2(0,T;L^2(\Omega_p))}
+\|\vec{\ul}^{h,1}\|^2_0+\|\vec{\ul}^{h,0}\|^2_0+\|\vec{\ul}^{h,1}\|_0\|\vec{\ul}^{h,0}\|_0\big),
\end{align*}
where
%$k=\max \limits_{1\leq m\leq N-1}\{k_m\}$
C is a positive constant, which is independent of h, $k_m$ or other parameters and $\tau_{min}\leq\tau\leq\tau_{max}$.
\end{theorem}

\begin{proof}
From (\ref{ctf}), we get
$$\hat{\vec{\ul}}^{h,m+1}=\frac{2(1-\theta)\tau_{m-1}+1}{\tau_{m-1}+1}\vec{\ul}^{h,m+1}
-(1-2\theta)\tau_{m-1}\vec{\ul}^{h,m}
+\frac{(1-2\theta)\tau_{m-1}^2}{\tau_{m-1}+1}\vec{\ul}^{h,m-1},$$
$$\hat{p}_{f}^{h,m+1}=\frac{2(1-\theta)\tau_{m-1}+1}{\tau_{m-1}+1}p_{f}^{h,m+1}
-(1-2\theta)\tau_{m-1}p_{f}^{h,m}
+\frac{(1-2\theta)\tau_{m-1}^2}{\tau_{m-1}+1}p_{f}^{h,m-1},$$
and then take them into (\ref{clm}), for $\forall\ \vec{\vl}^{h}\in X_h, \ q_{f}^{h}\in Q_{fh}$,
\begin{align}\label{coupledsum}
&\frac{1}{k_m}\big(A(\vec{\ul}^{h,m+1}),\vec{\vl}^{h}\big)+a\big(B(\vec{\ul}^{h,m+1}),\vec{\vl}^{h}\big)
+b\big(\vec{\vl}^{h},B(p_{f}^{h,m+1})\big)
=\big<(1-\theta)\vec{F}^{m+1}+\theta\vec{F}^{m},\vec{\vl}^{h}\big>,\\
&b\big(B(\vec{\ul}^{h,m+1}), q_{f}^{h}\big)=0.\nonumber
\end{align}

Setting $\vec{\vl}^{h}=2k_mB(\vec{\ul}^{h,m+1})$, $q_{f}^{h}=2k_mB(p_{f}^{h,m+1})$
and analysing each term of the first equation in (\ref{coupledsum}). First since $\frac{2(1-\theta)^2\tau_{m-1}+1-\theta}{\tau_{m-1}+1}
-\frac{1}{2}\frac{2(1-\theta)\tau_{m-1}+1}{\tau_{m-1}+1}>0$,
we can use Lemma \ref{lemma} to handle the first term on the left-hand side
\begin{align}
\begin{split}\label{cl1}
&2\big(A(\vec{\ul}^{h,m+1}),B(\vec{\ul}^{h,m+1})\big)\\
\geq&
D(\theta,\tau_{m-1})\|\vec{\ul}^{h,m+1}\|^2_0
-H(\theta,\tau_{m-1})\|\vec{\ul}^{h,m}\|^2_0
-E(\theta,\tau_{m-1})\|\vec{\ul}^{h,m-1}\|^2_0\\
&-F(\theta,\tau_{m-1})\big(\|\vec{\ul}^{h,m+1}\|_0\|\vec{\ul}^{h,m}\|_0
-\|\vec{\ul}^{h,m}\|_0\|\vec{\ul}^{h,m-1}\|_0\big)\\
=&
C_{D}D(\theta,\tau)\|\vec{\ul}^{h,m+1}\|^2_0
-C_{H}H(\theta,\tau)\|\vec{\ul}^{h,m}\|^2_0
-C_{E}E(\theta,\tau)\|\vec{\ul}^{h,m-1}\|^2_0\\
&-C_{F}F(\theta,\tau)\big(\|\vec{\ul}^{h,m+1}\|_0\|\vec{\ul}^{h,m}\|_0
-\|\vec{\ul}^{h,m}\|_0\|\vec{\ul}^{h,m-1}\|_0\big)\\
\geq&
C_{min}\bigg(D(\theta,\tau)\|\vec{\ul}^{h,m+1}\|^2_0
-H(\theta,\tau)\|\vec{\ul}^{h,m}\|^2_0
-E(\theta,\tau)\|\vec{\ul}^{h,m-1}\|^2_0\\
&-F(\theta,\tau)\big(\|\vec{\ul}^{h,m+1}\|_0\|\vec{\ul}^{h,m}\|_0
-\|\vec{\ul}^{h,m}\|_0\|\vec{\ul}^{h,m-1}\|_0\big)\bigg).
\end{split}
\end{align}
where $C_{min}=\min\{C_{D},C_{H},C_{E},C_{F}\}$, ($C_{i}=\frac{i(\theta,\tau_{m-1})}{i(\theta,\tau)}>0,i=D,H,E,F.$) and $\tau_{min}\leq\tau\leq\tau_{max}$.

Then using the coercivity of the bilinear form $a(\cdot,\cdot)$, the second term on the left can be handled as
\begin{align}\label{cl2}
2k_ma\big(B(\vec{\ul}^{h,m+1}),B(\vec{\ul}^{h,m+1})\big)\geq2C_{coe}k_m\|B(\vec{\ul}^{h,m+1})\|_X^2.
\end{align}
Finally, we use the Cauchy Schwarz inequality and Young's inequality,
the external force term on the right can be written as
\begin{align}\label{cr1}
&2k_m\big<(1-\theta)\vec{F}^{m+1}+\theta\vec{F}^{m},B(\vec{\ul}^{h,m+1})\big>\\
&\leq
\frac{\theta^2k_m}{C_{coe}}\|\vec{F}^{m}\|_{X'}^2
+\frac{(1-\theta)^2k_m}{C_{coe}} \|\vec{F}^{m+1}\|_{X'}^2+C_{coe}k_m\|B(\vec{\ul}^{h,m+1})\|_X^2.\nonumber
\end{align}

Combining the (\ref{cl1})-(\ref{cr1}) and sum them over $m=1,2,...,N-1$, and let $k=\max \limits_{1\leq m\leq N-1}\{k_m\}$, we have
\begin{align*}
&C_{min}\bigg(D(\theta,\tau)\|\vec{\ul}^{h,N}\|^2_0
+E(\theta,\tau)\|\vec{\ul}^{h,N-1}\|^2_0
-F(\theta,\tau)\|\vec{\ul}^{h,N}\|_0\|\vec{\ul}^{h,N-1}\|_0\bigg)
+C_{coe}\sum_{m=1}^{N-1}\big[k_m\|B(\vec{\ul}^{h,m+1})\|_X^2\big]\\
\leq&\frac{k\theta^2}{C_{coe}}\sum_{m=1}^{N-1}\|\vec{F}^{m}\|_{X'}^2
    +\frac{k(1-\theta)^2}{C_{coe}}\sum_{m=1}^{N-1}\|\vec{F}^{m+1}\|_{X'}^2\\
&+C_{min}\bigg(D(\theta,\tau)\|\vec{\ul}^{h,1}\|^2_0
+E(\theta,\tau)\|\vec{\ul}^{h,0}\|^2_0
-F(\theta,\tau)\|\vec{\ul}^{h,1}\|_0\|\vec{\ul}^{h,0}\|_0\bigg).\nonumber
\end{align*}
Note that
\begin{align*}
-F(\theta,\tau)\|\vec{\ul}^{h,N}\|_0\|\vec{\ul}^{h,N-1}\|_0
\geq-G(\theta,\tau)\|\vec{\ul}^{h,N}\|_0^2-E(\theta,\tau)\|\vec{\ul}^{h,N-1}\|^2_0,
\end{align*}
so we have
\begin{align*}
&C_{min}I(\theta,\tau)\|\vec{\ul}^{h,N}\|^2_0
+C_{coe}\sum_{m=1}^{N-1}\big[k_m\|B(\vec{\ul}^{h,m+1})\|_X^2\big]\\
\leq&\frac{k\theta^2}{C_{coe}}\sum_{m=1}^{N-1}\|\vec{F}^{m}\|_{X'}^2
    +\frac{k(1-\theta)^2}{C_{coe}}\sum_{m=1}^{N-1}\|\vec{F}^{m+1}\|_{X'}^2\\
&+C_{min}\bigg(D(\theta,\tau)\|\vec{\ul}^{h,1}\|^2_0
+E(\theta,\tau)\|\vec{\ul}^{h,0}\|^2_0
-F(\theta,\tau)\|\vec{\ul}^{h,1}\|_0\|\vec{\ul}^{h,0}\|_0\bigg).
\end{align*}
Thus, we end the proof.
\end{proof}

Then, we derive following stability theorem of decoupled variable time-stepping algorithm(\textbf{Algorthm 2}).
\begin{theorem}\label{theorem2} (Stability of \textbf{Algorthm 2})
Let $\vec{u}_{f}^{h,m+1}$, $\phi_{p}^{h,m+1}$ be the solution of the Linear Multi-step methods plus time filter with $0<\theta<\frac{1}{2}$. For $N\geq 2$, we have
\begin{align*}
&I(\theta,\tau)\|\vec{u}_{f}^{h,N}\|^2_{f}+gI(\theta,\tau)\|\phi_p^{h,N}\|^2_{p}
+\hat{C}_{coe}\sum_{m=1}^{N-1}\big[k_m\|B(\vec{u}_{f}^{h,m+1})\|_{X_f}^2\big]
+g\check{C}_{coe}\sum_{m=1}^{N-1}\big[k_m\|B(\phi_{p}^{h,m+1})\|_{X_p}^2\big]\\
\leq &C(T)
\bigg(\|\vec{g}_{f}\|_{L^2(0,T;L^2(\Omega_f))}^2+\|g_{p}\|_{L^2(0,T;L^2(\Omega_p))}^2\\
&+C_{min}\big(D(\theta,\tau)\|\vec{u}_{f}^{h,1}\|^2_f
+E(\theta,\tau)\|\vec{u}_{f}^{h,0}\|^2_f
-F(\theta,\tau)\|\vec{u}_{f}^{h,1}\|_f\|\vec{u}_{f}^{h,0}\|_f\big)\\
&+gC_{min}\big(D(\theta,\tau)\|\phi_{p}^{h,1}\|^2_p
+E(\theta,\tau)\|\phi_{p}^{h,0}\|^2_p
-F(\theta,\tau)\|\phi_{p}^{h,1}\|_p\|\phi_{p}^{h,0}\|_p\big)\bigg).
\end{align*}
where
%$k=\max \limits_{1\leq m\leq N-1}\{k_m\}$ and
$C(T)=exp\left(\sum \limits_{m=1}^{N-1}\max\left\{\frac{2k}{\hat{C}_{coe}C_{min}I(\theta,\tau)},
\frac{2k}{\check{C}_{coe}C_{min}I(\theta,\tau)}
\right\}\right)$ and $\tau_{min}\leq\tau\leq\tau_{max}$.
\end{theorem}
\begin{proof}
From (\ref{dtf}), we have
\begin{align}
&\hat{\vec{u}}_{f}^{h,m+1}=\frac{2(1-\theta)\tau_{m-1}+1}{\tau_{m-1}+1}{\vec{u}}_{f}^{h,m+1}
-(1-2\theta)\tau_{m-1}{\vec{u}}_{f}^{h,m}
+\frac{(1-2\theta)\tau_{m-1}^2}{\tau_{m-1}+1}{\vec{u}}_{f}^{h,m-1},\nonumber\\
&\hat{p}_{f}^{h,m+1}=\frac{2(1-\theta)\tau_{m-1}+1}{\tau_{m-1}+1}p_{f}^{h,m+1}
-(1-2\theta)\tau_{m-1}p_{f}^{h,m}
+\frac{(1-2\theta)\tau_{m-1}^2}{\tau_{m-1}+1}p_{f}^{h,m-1},\nonumber\\
&\hat{\phi}_{p}^{h,m+1}=\frac{2(1-\theta)\tau_{m-1}+1}{\tau_{m-1}+1}\phi_{p}^{h,m+1}
-(1-2\theta)\tau_{m-1}\phi_{p}^{h,m}
+\frac{(1-2\theta)\tau_{m-1}^2}{\tau_{m-1}+1}\phi_{p}^{h,m-1},\nonumber
\end{align}
then take them into (\ref{ds}) and (\ref{dd}), add them together,
for $\forall\ \vec{v}_{f}^{h}\in X_{fh}, \psi_{p}^{h}\in X_{ph}$ and $\ q_{f}^{h}\in Q_{fh}$,
\begin{align}\label{decoupledsum}
&\frac{1}{k_m}\big(A(\vec{u}_{f}^{h,m+1}),\vec{v}_{f}^{h}\big)_{\Omega_f}
+\frac{g}{k_m}\big(A(\phi_{p}^{h,m+1}),\psi_{p}^{h}\big)_{\Omega_p}\nonumber\\
&+a_{\Omega_f}\big(B(\vec{u}_{f}^{h,m+1}),\vec{v}_{f}^{h}\big)
+a_{\Omega_p}\big(B(\phi_{p}^{h,m+1}),\psi_{p}^{h}\big)
+b\big(\vec{v}_{f}^{h},B(p_{f}^{h,m+1})\big)\nonumber\\
=&\big<(1-\theta)\vec{g}_{f}^{m+1}+\theta\vec{g}_{f}^{m},\vec{v}_{f}^{h}\big>_{\Omega_f}
+g\big<(1-\theta)g_{p}^{m+1}+\theta g_{p}^{m},\psi_{p}^{h}\big>_{\Omega_p}\\
&-c_{\Gamma}\big(\vec{v}_{f}^{h},(1+(1-\theta)\tau_{m-1})\phi_{p}^{h,m}
-(1-\theta)\tau_{m-1}\phi_{p}^{h,m-1}\big)\nonumber\\
&+c_{\Gamma}\big((1+(1-\theta)\tau_{m-1})\vec{u}_{f}^{h,m}-(1-\theta)\tau_{m-1}\vec{u}_{f}^{h,m-1},
\psi_{p}^{h}\big),\nonumber\\
&b\big(B(\vec{u}_{f}^{h,m+1}),q_{f}^{h}\big)=0.\nonumber
\end{align}

Setting $\vec{v}_{f}^{h}=2k_m B(\vec{u}_{f}^{h,m+1})$, $\psi_{p}^{h}=2k_m B(\phi_{p}^{h,m+1})$ and $q_{f}^{h}=2k_mB(p_{f}^{h,m+1})$.
Looking back at the proof process of Theorem \ref{theorem1}, we have the following equations to hold
\begin{align}\label{dl1}
&2\big(A(\vec{u}_{f}^{h,m+1}),B(\vec{u}_{f}^{h,m+1})\big)_{\Omega_f}\nonumber\\
\geq&
C_{min}\bigg(D(\theta,\tau)\|\vec{u}_{f}^{h,m+1}\|^2_f
-H(\theta,\tau)\|\vec{u}_{f}^{h,m}\|^2_f
-E(\theta,\tau)\|\vec{u}_{f}^{h,m-1}\|^2_f\\
&-F(\theta,\tau)(\|\vec{u}_{f}^{h,m}\|_f\|\vec{u}_{f}^{h,m-1}\|_f
-\|\vec{u}_{f}^{h,m}\|_f\|\vec{u}_{f}^{h,m-1}\|_f)\bigg),\nonumber
\end{align}
and
\begin{align}\label{dl2}
&2g\big(A(\phi_{p}^{h,m+1}),B(\phi_{p}^{h,m+1})\big)_{\Omega_p}\nonumber\\
\geq&gC_{min}\bigg(D(\theta,\tau)\|\phi_{p}^{h,m+1}\|^2_p
-H(\theta,\tau)\|\phi_{p}^{h,m}\|^2_p
-E(\theta,\tau)\|\phi_{p}^{h,m-1}\|^2_p\\
&-F(\theta,\tau)(\|\phi_{p}^{h,m}\|_p\|\phi_{p}^{h,m-1}\|_p
-\|\phi_{p}^{h,m}\|_p\|\phi_{p}^{h,m-1}\|_p)\bigg),\nonumber
\end{align}
and
\begin{align}\label{dl3}
2k_m a_{\Omega_f}\big(B(\vec{u}_{f}^{h,m+1}),B(\vec{u}_{f}^{h,m+1})\big)
\geq2\hat{C}_{coe}k_m\|B(\vec{u}_{f}^{h,m+1})\|_{X_f}^2,
\end{align}
and
\begin{align}\label{dl4}
2k_ma_{\Omega_p}\big(B(\phi_{p}^{h,m+1}),B(\phi_{p}^{h,m+1})\big)
\geq2g\check{C}_{coe}k_m\|B(\phi_{p}^{h,m+1})\|_{X_p}^2,
\end{align}
and
\begin{align}\label{dr1}
&2k_m\big<(1-\theta)\vec{g}_{f}^{m+1}+\theta\vec{g}_{f}^{m},B(\vec{u}_{f}^{h,m+1})\big>\\
\leq&
\frac{2\theta^2k_m}{\hat{C}_{coe}}\|\vec{g}_{f}^{m}\|_{X_f'}^2
+\frac{2(1-\theta)^2k_m}{\hat{C}_{coe}} \|\vec{g}_{f}^{m+1}\|_{X_f'}^2
+\frac{\hat{C}_{coe}k_m}{2}\|B(\vec{u}_{f}^{h,m+1})\|_{X_f}^2,\nonumber
\end{align}
and
\begin{align}\label{dr2}
&2gk_m\big<(1-\theta)g_{p}^{m+1}+\theta g_{p}^{m},B(\phi_{p}^{h,m+1})\big>\\
&\leq
\frac{2g\theta^2k_m}{\check{C}_{coe}}\|g_{p}^{m}\|_{X_p'}^2
+\frac{2g(1-\theta)^2k_m}{\check{C}_{coe}} \|g_{p}^{m+1}\|_{X_p'}^2
+\frac{g\check{C}_{coe}k_m}{2}B(\phi_{p}^{h,m+1})\|_{X_p}^2.\nonumber
\end{align}

Here, we need to analyze the interface item on the right-hand side. Using the Lemma \ref{gamma} presented earlier, and bringing in the appropriate parameters $\varepsilon_1=\frac{1}{\hat{C}_{coe}}$ and $\varepsilon_2=\frac{g}{\check{C}_{coe}}$, we have
\begin{align}\label{dr3}
&-c_{\Gamma}\big(B(\vec{u}_{f}^{h,m+1}),
(1+(1-\theta)\tau_{m-1})\phi_{p}^{h,m}-(1-\theta)\tau_{m-1}\phi_{p}^{h,m-1}\big)\nonumber\\
&+c_{\Gamma}\big((1+(1-\theta)\tau_{m-1})\vec{u}_{f}^{h,m}-(1-\theta)\tau_{m-1}\vec{u}_{f}^{h,m-1},
B(\phi_{p}^{h,m+1})\big)\\
\leq &\frac{\hat{C}_{coe}k_m}{2}\|B(\vec{u}_{f}^{h,m+1})\|_{X_f}^2
+\frac{2C_1h^{-1}k_m}{\hat{C}_{coe}}
\|(1+(1-\theta)\tau_{m-1})\phi_{p}^{h,m}-(1-\theta)\tau_{m-1}\phi_{p}^{h,m-1}\|_p^2\nonumber\\
&+\frac{g\check{C}_{coe}k_m}{2}\|B(\phi_{p}^{h,m+1})\|_{X_p}^2
+\frac{2gC_2h^{-1}k_m}{\check{C}_{coe}}
\|(1+(1-\theta)\tau_{m-1})\vec{u}_{f}^{h,m}-(1-\theta)\tau_{m-1}\vec{u}_{f}^{h,m-1}\|_f^2\nonumber\\
&\leq
\frac{\hat{C}_{coe}k_m}{2}\|B(\vec{u}_{f}^{h,m+1})\|_{X_f}^2
+\frac{2gk_m}{\hat{C}_{coe}}
\|(1+(1-\theta)\tau_{m-1})\phi_{p}^{h,m}-(1-\theta)\tau_{m-1}\phi_{p}^{h,m-1}\|_p^2\nonumber\\
&+\frac{g\check{C}_{coe}k_m}{2}\|B(\phi_{p}^{h,m+1})\|_{X_p}^2
+\frac{2k_m}{\check{C}_{coe}}
\|(1+(1-\theta)\tau_{m-1})\vec{u}_{f}^{h,m}-(1-\theta)\tau_{m-1}\vec{u}_{f}^{h,m-1}\|_f^2,\nonumber
\end{align}
where the last inequality follows from properly chosen constant $C_1$ and $C_2$.

Combining the (\ref{dl1})-(\ref{dr3}), and sum (\ref{decoupledsum}) over $m=1,2,...,N-1$, we  use the same method as the Theorem \ref{theorem1}, and let $k=\max \limits_{1\leq m\leq N-1}\{k_m\}$, we have
\begin{align*}
&C_{min}\bigg(D(\theta,\tau)\|\vec{u}_{f}^{h,N}\|^2_f
+E(\theta,\tau)\|\vec{u}_{f}^{h,N-1}\|^2_f
-F(\theta,\tau)\|\vec{u}_{f}^{h,N}\|_f\|\vec{u}_{f}^{h,N-1}\|_f\bigg)\\
&+gC_{min}\bigg(D(\theta,\tau)\|\phi_p^{h,N}\|^2_p
+E(\theta,\tau)\|\phi_p^{h,N-1}\|^2_p
-F(\theta,\tau)\|\phi_p^{h,N}\|_p\|\phi_p^{h,N-1}\|_p\bigg)\\
&+\hat{C}_{coe}\sum_{m=1}^{N-1}\big[k_m\|B(\vec{u}_{f}^{h,m+1})\|_{X_f}^2\big]
+g\check{C}_{coe}\sum_{m=1}^{N-1}\big[k_m\|B(\phi_{p}^{h,m+1})\|_{X_p}^2\big]\\
\leq&\frac{2k\theta^2}{\hat{C}_{coe}}\sum_{m=1}^{N-1}\|\vec{g}_{f}^{m}\|_{X_f'}^2
    +\frac{2k(1-\theta)^2}{\hat{C}_{coe}}\sum_{m=1}^{N-1}\|\vec{g}_{f}^{m+1}\|_{X_f'}^2
    +\frac{2gk\theta^2}{\check{C}_{coe}}\sum_{m=1}^{N-1}\|g_{p}^{m}\|_{X_p'}^2
    +\frac{2gk(1-\theta)^2}{\check{C}_{coe}}\sum_{m=1}^{N-1}\|g_{p}^{m+1}\|_{X_p'}^2\\
&+\frac{2gk}{\hat{C}_{coe}}\sum_{m=1}^{N-1}
\|(1+(1-\theta)\tau_{m-1})\phi_{p}^{h,m}-(1-\theta)\tau_{m-1}\phi_{p}^{h,m-1}\|_p^2\\
&+\frac{2k}{\check{C}_{coe}}\sum_{m=1}^{N-1}
\|(1+(1-\theta)\tau_{m-1})\vec{u}_{f}^{h,m}-(1-\theta)\tau_{m-1}\vec{u}_{f}^{h,m-1}\|_f^2\\
&+C_{min}\bigg(D(\theta,\tau)\|\vec{u}_{f}^{h,1}\|^2_f
+E(\theta,\tau)\|\vec{u}_{f}^{h,0}\|^2_f
-F(\theta,\tau)\|\vec{u}_{f}^{h,1}\|_f\|\vec{u}_{f}^{h,0}\|_f\bigg)\\
&+gC_{min}\bigg(D(\theta,\tau)\|\phi_{p}^{h,1}\|^2_p
+E(\theta,\tau)\|\phi_{p}^{h,0}\|^2_p
-F(\theta,\tau)\|\phi_{p}^{h,1}\|_p\|\phi_{p}^{h,0}\|_p\bigg).
\end{align*}
Then note that
\begin{align*}
-F(\theta,\tau)\|\vec{u}_{f}^{h,N}\|_f\|\vec{u}_{f}^{h,N-1}\|_f
\geq-G(\theta,\tau)\|\vec{u}_{f}^{h,N}\|_f^2
-E(\theta,\tau)\|\vec{u}_{f}^{h,N-1}\|^2_f
\end{align*}
and
\begin{align*}
-F(\theta,\tau)\|\phi_p^{h,N}\|_p\|\phi_p^{h,N-1}\|_p
\geq-G(\theta,\tau)\|\phi_p^{h,N}\|_p^2
-E(\theta,\tau)\|\phi_p^{h,N-1}\|^2_p,
\end{align*}
so we can get
\begin{align*}
&C_{min}I(\theta,\tau)\|\vec{u}_{f}^{h,N}\|^2_{f}+gC_{min}I(\theta,\tau)\|\phi_p^{h,N}\|^2_{p}
+\hat{C}_{coe}\sum_{m=1}^{N-1}\big[k_m\|B(\vec{u}_{f}^{h,m+1})\|_{X_f}^2\big]
+g\check{C}_{coe}\sum_{m=1}^{N-1}\big[k_m\|B(\phi_{p}^{h,m+1})\|_{X_p}^2\big]\\
\leq&\frac{2k\theta^2}{\hat{C}_{coe}}\sum_{m=1}^{N-1}\|\vec{g}_{f}^{m}\|_{X_f'}^2
    +\frac{2k(1-\theta)^2}{\hat{C}_{coe}}\sum_{m=1}^{N-1}\|\vec{g}_{f}^{m+1}\|_{X_f'}^2
    +\frac{2k\theta^2}{\check{C}_{coe}}\sum_{m=1}^{N-1}\|g_{p}^{m}\|_{X_p'}^2
    +\frac{2k(1-\theta)^2}{\check{C}_{coe}}\sum_{m=1}^{N-1}\|g_{p}^{m+1}\|_{X_p'}^2\\
&+\frac{2k}{\hat{C}_{coe}C_{min}I(\theta,\tau)}
\sum_{m=1}^{N-1}gC_{min}I(\theta,\tau)\|\phi_{p}^{h,m}\|_p^2
+\frac{2k}{\check{C}_{coe}C_{min}I(\theta,\tau)}
\sum_{m=1}^{N-1}C_{min}I(\theta,\tau)\|\vec{u}_{f}^{h,m}\|_f^2\\
&+C_{min}\bigg(D(\theta,\tau)\|\vec{u}_{f}^{h,1}\|^2_f
+E(\theta,\tau)\|\vec{u}_{f}^{h,0}\|^2_f
-F(\theta,\tau)\|\vec{u}_{f}^{h,1}\|_f\|\vec{u}_{f}^{h,0}\|_f\bigg)\\
&+gC_{min}\bigg(D(\theta,\tau)\|\phi_{p}^{h,1}\|^2_p
+E(\theta,\tau)\|\phi_{p}^{h,0}\|^2_p
-F(\theta,\tau)\|\phi_{p}^{h,1}\|_p\|\phi_{p}^{h,0}\|_p\bigg).
\end{align*}
Applying the discrete Gronwall inequality, we end the proof.
\end{proof}
\section{Error estimates \label{4}}
In this section, we analyze the errors of coupled and decoupled variable time-stepping algorithms. For the sake of the later analysis, we define error functions:
\begin{eqnarray*}
&&e_{\vec{\ul}}^{h,m}=\vec{\ul}^{h,m}-\vec{\ul}^{m}=\vec{\ul}^{h,m}-P_h^{\vec{\ul}}\vec{\ul}^{m}
+P_h^{\vec{\ul}}\vec{\ul}^{m}-\vec{\ul}^{m}=\eta_{\vec{\ul}}^{h,m}+\zeta_{\vec{\ul}}^{h,m},\\
&&e_{p_f}^{h,m}=p_{f}^{h,m}-p_{f}^{m}=p_{f}^{h,m}-P_h^{p_f}p_{f}^{m}+P_h^{p}p_{f}^{m}-p_{f}^{m}
=\eta_{p}^{h,m}+\zeta_{p}^{h,m},\\
&&e_{\vec{u}_{f}}^{h,m}=\vec{u}_{f}^{h,m}-\vec{u}_{f}^{m}=\vec{u}_{f}^{h,m}-P_h^{\vec{u}_{f}}\vec{u}_{f}^{m}
+P_h^{\vec{u}_{f}}\vec{u}_{f}^{m}-\vec{u}_{f}^{m}=\eta_{\vec{u}_{f}}^{h,m}+\zeta_{\vec{u}_{f}}^{h,m},\\
&&e_{\phi_{p}}^{h,m}=\phi_{p}^{h,m}-\phi_{p}^{m}=\phi_{p}^{h,m}-P_h^{\phi_p}\phi_{p}^{m}+P_h^{\phi_p}\phi_{p}^{m}-\phi_{p}^{m}
=\eta_{\phi_{p}}^{h,m}+\zeta_{\phi_{p}}^{h,m}.
\end{eqnarray*}
Obviously, we have
\begin{eqnarray}
\begin{split}\label{theta}
&\|\zeta_{\vec{\ul}}^{h,m}\|_0\leq Ch^2\|\vec{\ul}^{h}(t)\|_{H^2},\quad
\|\zeta_{\vec{\ul}}^{h,m}\|_X\leq Ch\|\vec{\ul}^{h}(t)\|_{H^2},\quad
\|\zeta_{p_f}^{h,m}\|_{L^2}\leq Ch\|p_f^{h}(t)\|_{H^1},\\
&\|\zeta_{\vec{u}_f}^{h,m}\|_{f}\leq Ch^2\|\vec{u}_f^{h}(t)\|_{H^2},\quad
\|\zeta_{\vec{u}_f}^{h,m}\|_{X_f}\leq Ch\|\vec{u}_f^{h}(t)\|_{H^2},\\
&\|\zeta_{\phi_p}^{h,m}\|_{p}\leq Ch^2\|\phi_p^{h}(t)\|_{H^2},\quad
\|\zeta_{\phi_p}^{h,m}\|_{X_p}\leq Ch\|\phi_p^{h}(t)\|_{H^2}.
\end{split}
\end{eqnarray}
Note that $\eta_{\vec{\ul}}^{h,0}=0$, $\eta_{p}^{h,0}=0$, $\eta_{\vec{u}_f}^{h,0}=0$ and $\eta_{\phi_p}^{h,0}=0$.

Assume the solution satisfies the following regularity conditions:
\begin{eqnarray}
\begin{split}
\label{regularity}
&\vec{u}_f\in L^{\infty}(0,T;X_f^2),
\vec{u}_{f,t}\in L^{2}(0,T;X_f^1)\cap L^{\infty}(0,T;L^2),
\vec{u}_{f,tt}\in L^{2}(0,T;L^2), \vec{u}_{f,ttt}\in L^{2}(0,T;X_f'),\\
&\phi_p\in L^{\infty}(0,T;X_p^2),
\phi_{p,t}\in L^{2}(0,T;X_p^1)\cap L^{\infty}(0,T;L^2),
\phi_{p,tt}\in L^{2}(0,T;L^2),
\phi_{p,ttt}\in L^{2}(0,T;X_p').\\
\end{split}
\end{eqnarray}
And the external force $\vec{g}_f$ and $g_p$ also need to be satisfied
\begin{eqnarray}
\begin{split}\label{regularity1}
&\vec{g}_{f,t}\in L^{2}(0,T;L^2),\ \vec{g}_{f,tt}\in L^{2}(0,T;X_f'),
g_{p,t}\in L^{2}(0,T;L^2),\ g_{p,tt}\in L^{2}(0,T;X_p').
\end{split}
\end{eqnarray}

First, we derive following error estimate of coupled variable time-stepping algorithm (\textbf{Algorthm 1}).
\begin{theorem}\label{theorem3}(Second-order convergence of \textbf{Algorthm 1}) Under the assumption of (\ref{regularity}) and (\ref{regularity1}), for $N\geq 2$ we have the estimate
\begin{align*}
I(\theta,\tau)\|e_{\vec{\ul}}^{h,N}\|_0^2
+C_{coe}\sum_{m=1}^{N-1}\big[k_m\|B(e_{\vec{\ul}}^{h,m+1})\|_{X}^2\big]
\leq C (\tilde{k}^4+h^4),
\end{align*}
where $0<\theta<\frac{1}{2}$, $\tau_{min}\leq\tau\leq\tau_{max}$, $\tilde{k}=\max \limits_{1\leq m\leq N-1}\{k_m+k_{m-1}\}$ and $C$ is a positive constant.
\end{theorem}
\begin{proof}
First, let us multiply (\ref{coupled}) by $\frac{2(1-\theta)^2\tau_{m-1}+1-\theta}{\tau_{m-1}+1}$,
 $-((1-\theta)(1-2\theta)\tau_{m-1}-\theta)$ and $\frac{(1-\theta)(1-2\theta)\tau_{m-1}}{\tau_{m-1}+1}$ at $t_{m+1}$, $t_{m}$ and $t_{m-1}$, respectively. Then, subtracting the summation of three equations from (\ref{coupledsum}),
we use the error functions and the  properties of the project operator (\ref{projection}),
for $\forall\ \vec{\vl}^{h}\in X_h$ and $\ q_{f}^{h}\in Q_{f}^{h}$,
\begin{align}\label{ce}
&\frac{1}{k_m}\big(A(\eta_{\vec{\ul}}^{h,m+1}),\vec{\vl}^{h}\big)
+a\big(B(\eta_{\vec{\ul}}^{h,m+1}),\vec{\vl}^{h}\big)
+b\big(\vec{\vl}^{h}, B(\eta_{p}^{h,m+1})\big)\nonumber\\
=&-\big(\frac{1}{k_m}A(\vec{\ul}^{m+1})-B(\vec{\ul}_{t}^{m+1}),\vec{\vl}^{h}\big)
-\big(\frac{1}{k_m}A(\zeta_{\vec{\ul}}^{h,m+1}),\vec{\vl}^{h}\big)
+\big<S(\vec{F}^{m+1}),\vec{\vl}^{h}\big>,\\
&b\big(B(\eta_{\vec{\ul}}^{h,m+1}),q_{f}^{h}\big)=0.\nonumber
\end{align}

Setting $\vec{\vl}^{h}=2k_m B(\eta_{\vec{\ul}}^{h,m+1})$, $q_{f}^{h}=2k_m B(\eta_{p}^{h,m+1})$ and analysing each term of the first equation in (\ref{ce}).
Similar to Theorem \ref{theorem1}, we use Lemma \ref{lemma} to handle the first term on the left-hand side, we have
\begin{align}\label{cel1}
&2\big(A(\eta_{\vec{\ul}}^{h,m+1}), B(\eta_{\vec{\ul}}^{h,m+1})\big)\nonumber\\
\geq&C_{min}\bigg(D(\theta,\tau)\|\eta_{\vec{\ul}}^{h,m+1}\|^2_0
-H(\theta,\tau)\|\eta_{\vec{\ul}}^{h,m}\|^2_0
-E(\theta,\tau)\|\eta_{\vec{\ul}}^{h,m-1}\|^2_0\\
&-F(\theta,\tau)
(\|\eta_{\vec{\ul}}^{h,m+1}\|_0\|\eta_{\vec{\ul}}^{h,m}\|_0
-\|\eta_{\vec{\ul}}^{h,m}\|_0\|\eta_{\vec{\ul}}^{h,m-1}\|_0)\bigg).\nonumber
\end{align}
Then using the coercivity of the bilinear form $a(\cdot,\cdot)$, the second term on the left can be handled as
\begin{align}\label{cel2}
2k_m a\big(B(\eta_{\vec{\ul}}^{h,m+1}), B(\eta_{\vec{\ul}}^{h,m+1})\big)
\geq2C_{coe}k_m\|B(\eta_{\vec{\ul}}^{h,m+1})\|_{X}^2.
\end{align}

Next, we consider the right side of (\ref{ce}).
For the first term on the right-hand side, we use the Taylor expansion with the integral remainder,
\begin{align*}
&\vec{\ul}^{m}=\vec{\ul}^{m+1}-k_m\vec{\ul}^{m+1}_t+\frac{k_m^2}{2}\vec{\ul}^{m+1}_{tt}+\frac{1}{2}\int_{t^{m+1}}^{t^{m}}(t^{m}-t)^2\vec{\ul}_{ttt}dt,\\
&\vec{\ul}^{m-1}=\vec{\ul}^{m+1}-(k_m+k_{m-1})\vec{\ul}^{m+1}_t+\frac{(k_m+k_{m-1})^2}{2}\vec{\ul}^{m+1}_{tt}+\frac{1}{2}\int_{t^{m+1}}^{t^{m-1}}(t^{m-1}-t)^2\vec{\ul}_{ttt}dt,\\
&\vec{\ul}^{m}_t=\vec{\ul}^{m+1}_t-k_m\vec{\ul}^{m+1}_{tt}+\int_{t^{m+1}}^{t^{m}}(t^{m}-t)\vec{\ul}_{ttt}dt,\\
&\vec{\ul}^{m-1}_t=\vec{\ul}^{m+1}_t-(k_m+k_{m-1})\vec{\ul}^{m+1}_{tt}+\int_{t^{m+1}}^{t^{m-1}}(t^{m-1}-t)\vec{\ul}_{ttt}dt.
\end{align*}
So we get
\begin{align*}
&\frac{1}{k_m}A(\vec{\ul}^{m+1})-B(\vec{\ul}_{t}^{m+1})\\
=&\frac{(1-2\theta)\tau_{m-1}+1}{2k_{m}}\int_{t^{m}}^{t^{m+1}}(t^{m}-t)^2\vec{\ul}_{ttt}dt
-\frac{(1-2\theta)\tau_{m-1}^2}{2k_{m}(\tau_{m-1}+1)}\int_{t^{m-1}}^{t^{m+1}}(t^{m-1}-t)^2\vec{\ul}_{ttt}dt\\
&-((1-\theta)(1-2\theta)\tau_{m-1}-\theta)\int_{t^{m}}^{t^{m+1}}(t^{m}-t)\vec{\ul}_{ttt}dt
+\frac{(1-\theta)(1-2\theta)\tau_{m-1}^2}{\tau_{m-1}+1}\int_{t^{m-1}}^{t^{m+1}}(t^{m-1}-t)\vec{\ul}_{ttt}dt.
\end{align*}
By using Cauchy-Schwarz inequality,
\begin{align*}
&(\int_{t^{m}}^{t^{m+1}}(t^{m}-t)^2\vec{\ul}_{ttt}dt)^2\leq \frac{k_{m}^5}{5}\int_{t^{m}}^{t^{m+1}}\vec{\ul}_{ttt}^2dt,\\
&(\int_{t^{m-1}}^{t^{m+1}}(t^{m-1}-t)^2\vec{\ul}_{ttt}dt)^2 \leq \frac{(k_{m}+k_{m-1})^5}{5}\int_{t^{m-1}}^{t^{m+1}}\vec{\ul}_{ttt}^2dt,\\
&(\int_{t^{m}}^{t^{m+1}}(t^{m}-t)\vec{\ul}_{ttt}dt)^2\leq \frac{k_{m}^3}{3}\int_{t^{m}}^{t^{m+1}}\vec{\ul}_{ttt}^2dt,\\
&(\int_{t^{m-1}}^{t^{m+1}}(t^{m-1}-t)\vec{\ul}_{ttt}dt)^2 \leq \frac{(k_{m}+k_{m-1})^3}{3}\int_{t^{m-1}}^{t^{m+1}}\vec{\ul}_{ttt}^2dt,
\end{align*}
the first term on the right can be handled as
\begin{align}\label{cer1}
&2k_m\big(\frac{1}{k_m}A(\vec{\ul}^{m+1})-B(\vec{\ul}_{t}^{m+1}),B(\eta_{\vec{\ul}}^{h,m+1})\big)\nonumber\\
\leq& \frac{3k_m}{C_{coe}}\|\frac{1}{k_m}A(\vec{\ul}^{m+1})-B(\vec{\ul}_{t}^{m+1})\|_{X'}^{2}
+\frac{C_{coe}k_m}{3}\|B(\eta_{\vec{\ul}}^{h,m+1})\|_X^2\\
\leq&\frac{C_{coe}k_m}{3}\|B(\eta_{\vec{\ul}}^{h,m+1})\|_{X}^2
+\frac{(k_m+k_{m-1})^4}{C_{coe}}\bigg[\frac{(1-\theta)^2\tau_{m-1}^4+((1-2\theta)\tau_{m-1}+1)^2}{4}\nonumber\\
&+((1-\theta)(1-2\theta)\tau_{m-1}-\theta)^2+((1-\theta)(1-2\theta)\tau_{m-1}^2)^2\bigg]
      \int_{t^{m-1}}^{t^{m+1}}\|\vec{\ul}_{ttt}\|^2_{X'}dt.\nonumber
%\frac{(k_m+k_{m-1})^4((1-2\theta)^2\tau_{m-1}^4((1-\theta)^2+1)
%+(\tau_{m-1}+1)^2((1-2\theta)\tau_{m-1}+1)^2+((1-\theta)(1-2\theta)\tau_{m-1}-\theta)^2)}{C_{coe}(\tau_{m-1}+1)^2}
%\int_{t^{m-1}}^{t^{m+1}}\|\vec{\ul}_{ttt}\|^2_{X'}dt\nonumber\\
\end{align}

In the same way, for the second term on the right side, we use the Taylor expansion with the integral remainder,
\begin{align*}
	&\vec{\ul}^{m}=\vec{\ul}^{m+1}+\int_{t^{m+1}}^{t^{m}}\vec{\ul}_{t}dt,\\
	&\vec{\ul}^{m-1}=\vec{\ul}^{m+1}+\int_{t^{m+1}}^{t^{m-1}}\vec{\ul}_{t}dt.
\end{align*}
Then
\begin{align*}
\frac{1}{k_m}A(\zeta_{\vec{\ul}}^{h,m+1})
&=\frac{1}{k_m}\big[(P_h^{\vec{\ul}}-I)A(\vec{\ul}^{m+1})\big]\\
&=\frac{1}{k_m}\left[((1-2\theta)\tau_{m-1}+1)\int_{t^{m}}^{t^{m+1}}(P_h^{\vec{\ul}}-I)\vec{\ul}_{t}dt
-\frac{(1-2\theta)\tau_{m-1}^2}{\tau_{m-1}+1}\int_{t^{m-1}}^{t^{m+1}}(P_h^{\vec{\ul}}-I)\vec{\ul}_{t}dt\right].
\end{align*}
Thus, we have
\begin{align}\label{cer2}
&2k_m\big(\frac{1}{k_m}A(\zeta_{\vec{\ul}}^{h,m+1}),B(\eta_{\vec{\ul}}^{h,m+1})\big)\nonumber\\
\leq& \frac{C_{coe}k_m}{3}\|B(\eta_{\vec{\ul}}^{h,m+1})\|_X^2
+\frac{3k_m}{C_{coe}}\|\frac{1}{k_m}A(\zeta_{\vec{\ul}}^{h,m+1})\|^2_{X'}\\
\leq&\frac{C_{coe}k_m}{3}\|B(\eta_{\vec{\ul}}^{h,m+1})\|_X^2
+\frac{3((1-2\theta)\tau_{m-1}+1)^2+3(1-2\theta)^2\tau_{m-1}^4}{C_{coe}}
\int_{t^{m-1}}^{t^{m+1}}\|(P_h^{\vec{\ul}}-I)\vec{\ul}_t\|^2_{X'}dt.\nonumber
\end{align}

Simlarly, for the third term on the right,
\begin{align*}
   &\vec{F}^{m+1}=\vec{F}^{m}+k_m\vec{F}_{t}^{m}+\int_{t^{m}}^{t^{m+1}}(t^{m+1}-t)\vec{F}_{tt}dt,\\
   &\vec{F}^{m-1}=\vec{F}^{m}-k_{m-1}\vec{F}_{t}^{m}+\int_{t^{m}}^{t^{m-1}}(t^{m-1}-t)\vec{F}_{tt}dt,
\end{align*}
then
\begin{align*}
&S(\vec{F}^{m+1})=-\frac{(1-\theta)(1-2\theta)\tau_{m-1}}{\tau_{m-1}+1}\int_{t^m}^{t^{m+1}}(t^{m+1}-t)\vec{F}_{tt}dt
+\frac{(1-\theta)(1-2\theta)\tau_{m-1}^2}{\tau_{m-1}+1}\int_{t^{m-1}}^{t^{m}}(t^{m-1}-t)\vec{F}_{tt}dt.
\end{align*}
Thus, we have
\begin{align}\label{cer3}
&2k_m\big<S(\vec{F}^{m+1}),B(\eta_{\vec{\ul}}^{h,m+1})\big>\nonumber\\
\leq& \frac{C_{coe}k_m}{3}\|B(\eta_{\vec{\ul}}^{h,m+1})\|_X^2
+\frac{3k_m}{C_{coe}}\|S(\vec{F}^{m+1})\|_{X'}^2\\
\leq&\frac{C_{coe}k_m}{3}\|B(\eta_{\vec{\ul}}^{h,m+1})\|_X^2
+\frac{(k_m+k_{m-1})^4(1-\theta)^2(1-2\theta)^2(\tau_{m-1}^2+1)\tau_{m-1}^2}{C_{coe}(\tau_{m-1}+1)^2}\int_{t^{m-1}}^{t^{m+1}} \|\vec{F}_{tt}\|_{X'}^2dt.\nonumber
\end{align}

Combining the (\ref{cel1})-(\ref{cer3})  and sum the (\ref{ce}) over $m=1,2,...,N-1$, and we use the same method as Theorem \ref{theorem1}. Let $\tilde{k}=\max \limits_{1\leq m\leq N-1}\{k_m+k_{m-1}\}$, we have
\begin{align*}
&C_{min}I(\theta,\tau)\|\eta_{\vec{\ul}}^{h,N}\|_0^2
+C_{coe}\sum_{m=1}^{N-1}[k_m\|B(\eta_{\vec{\ul}}^{h,m+1})\|_{X}^2]\nonumber\\
\leq&
\sum_{m=1}^{N-1}\left[\frac{\tilde{k}^4}{C_{coe}}\bigg(\frac{(1-\theta)^2\tau_{m-1}^4+((1-2\theta)\tau_{m-1}+1)^2}{4}\right.\nonumber\\
      &\left.+((1-\theta)(1-2\theta)\tau_{m-1}-\theta)^2+((1-\theta)(1-2\theta)\tau_{m-1}^2)^2\bigg)
      \int_{t^{m-1}}^{t^{m+1}}\|\vec{\ul}_{ttt}\|^2_{X'}dt\right.\\
%\frac{(k_m+k_{m-1})^4((1-2\theta)^2\tau_{m-1}^4((1-\theta)^2+1)
%+(\tau_{m-1}+1)^2((1-2\theta)\tau_{m-1}+1)^2+((1-\theta)(1-2\theta)\tau_{m-1}-\theta)^2)}{C_{coe}(\tau_{m-1}+1)^2}
%\int_{t^{m-1}}^{t^{m+1}}\|\vec{\ul}_{ttt}\|^2_{X'}dt\nonumber\\
&\left.+\frac{\tilde{k}^4(1-\theta)^2(1-2\theta)^2(\tau_{m-1}^2+1)\tau_{m-1}^2}{C_{coe}(\tau_{m-1}+1)^2}
\int_{t^{m-1}}^{t^{m+1}} \|\vec{F}_{tt}\|_{X'}^2dt\right.\nonumber\\
&\left.+\frac{3((1-2\theta)\tau_{m-1}+1)^2+3(1-2\theta)^2\tau_{m-1}^4}{C_{coe}}
\int_{t^{m-1}}^{t^{m+1}}\|(P_h^{\vec{\ul}}-I)\vec{\ul}_t\|^2_{X'}dt\right].\nonumber
%+\frac{6(1-2\theta)^2\tau_{m-1}^4+12(1-\theta)(1-2\theta)\tau_{m-1}^3
%+3((1-2\theta)(5-2\theta)+1)\tau_{m-1}^2+12(1-\theta)\tau_{m-1}+3}{C_{coe}(\tau_{m-1}+1)^2} \int_{t^{m-1}}^{t^{m+1}}\|(P_h^{\vec{\ul}}-I)\vec{\ul}_t\|^2_{X'}dt).\nonumber
\end{align*}

Finally, using the triangle inequality, we end the proof.
\end{proof}

Then, we derive the following error estimate of decoupled variable time-stepping algorithm(\textbf{Algorithm 2}).
\begin{theorem}\label{theorem4}(Second-order convergence of \textbf{Algorithm 2}) Under the assumption of (\ref{regularity}) and (\ref{regularity1}), for $N\geq 2$ we have the estimate
\begin{align*}
I(\theta,\tau)\|\eta_{\vec{u}_f}^{h,N}\|_f^2
+gI(\theta,\tau)\|\eta_{\phi_p}^{h,N}\|_p^2&\\
+\hat{C}_{coe}\sum_{m=1}^{N-1}\big[k_m\|B(e_{\vec{u}_f}^{h,m+1})\|_{X_f}^2\big]
+g\check{C}_{coe}\sum_{m=1}^{N-1}\big[k_m\|B(e_{\phi_p}^{h,m+1})\|_{X_p}^2\big]&
\leq C (\tilde{k}^4+h^4),
\end{align*}
where $0<\theta<\frac{1}{2}$, $\tau_{min}\leq\tau\leq\tau_{max}$, $\tilde{k}=\max \limits_{1\leq m\leq N-1}\{k_m+k_{m-1}\}$ and $C$ is a positive constant.
\end{theorem}
\begin{proof}
First, let us multiply (\ref{coupled}) by $\frac{2(1-\theta)^2\tau_{m-1}+1-\theta}{\tau_{m-1}+1}$,
 $-((1-\theta)(1-2\theta)\tau_{m-1}-\theta)$ and $\frac{(1-\theta)(1-2\theta)\tau_{m-1}}{\tau_{m-1}+1}$ at $t_{m+1}$, $t_{m}$ and $t_{m-1}$, respectively. Then, subtracting the summation of three equations from (\ref{decoupledsum}),
we use the error functions and the  properties of the project operator (\ref{projection}),
for $\forall\ \vec{v}_{f}^{h}\in X_{fh}, \phi_{p}^{h}\in X_{ph}$ and $\ q_{f}^{h}\in Q_{fh}$,
\begin{align}\label{de}
&\frac{1}{k_m}\big(A(\eta_{\vec{u}_f}^{h,m+1}),\vec{v}_{f}^{h}\big)_{\Omega_f}
+g\frac{1}{k_m}\big(A(\eta_{\phi_p}^{h,m+1}),\psi_{p}^{h}\big)_{\Omega_p}\nonumber\\
&+a_{\Omega_f}\big(B(\eta_{\vec{u}_f}^{h,m+1}),\vec{v}_{f}^{h}\big)
+a_{\Omega_p}\big(B(\eta_{\phi_p}^{h,m+1}),\psi_{p}^{h}\big)
+b\big(\vec{v}_{f}^{h}, B(\eta_{p_f}^{h,m+1})\big)\nonumber\\
=&-\big(\frac{1}{k_m}A(\vec{u}_{f}^{m+1})-B(\vec{u}_{f,t}^{m+1}),
\vec{v}_{f}^{h}\big)_{\Omega_f}
-g\big(\frac{1}{k_m}A(\phi_{p}^{m+1})-b(\phi_{p,t}^{m+1}),
\psi_{p}^{h}\big)_{\Omega_p}\\
&+\big(\frac{1}{k_m}A(\zeta_{\vec{u}_{f}}^{h,m+1}),
\vec{v}_{f}^{h}\big)_{\Omega_f}
+g\big(\frac{1}{k_m}A(\zeta_{\phi_{p}}^{h,m+1}),
\psi_{p}^{h}\big)_{\Omega_p}\nonumber\\
&+\big<S(\vec{g}_{f}^{m+1}),\vec{v}_{f}^{h}\big>_{\Omega_f}
+g\big<S(g_{p}^{m+1}),\psi_{p}^{h}\big>_{\Omega_p}\nonumber\\
&-c_{\Gamma}\big(\vec{v}_{f}^{h},
(1+(1-\theta)\tau_{m-1})\phi_{p}^{h,m}-(1-\theta)\tau_{m-1}\phi_{p}^{h,m-1}
-B(\phi_{p}^{m+1})\big)\nonumber\\
&+c_{\Gamma}\big((1+(1-\theta)\tau_{m-1})\vec{u}_{f}^{h,m}-(1-\theta)\tau_{m-1}\vec{u}_{f}^{h,m-1}
-B(\vec{u}_{f}^{m+1}),\psi_{p}^{h}\big),\nonumber\\
&b\big(B(\eta_{\vec{u}_f}^{h,m+1}),q_{f}^{h}\big)=0.\nonumber
\end{align}

Setting $\vec{v}_{f}^{h}=2k_m B(\eta_{\vec{u}_f}^{h,m+1})$, $\psi_{p}^{h}=2k_mB(\eta_{\phi_p}^{h,m+1})$ and $q_{f}^{h}=2k_m B(\eta_{p}^{h,m+1})$.
Review the proof process of the Theorem \ref{theorem3}, the following equations hold
 \begin{align}\label{del11}
&2\big(A(\eta_{\vec{u}_f}^{h,m+1}),B(\eta_{\vec{u}_f}^{h,m+1})\big)_{\Omega_f}\nonumber\\
\geq&
C_{min}\bigg(D(\theta,\tau)\|\eta_{\vec{u}_f}^{h,m+1}\|^2_f
-H(\theta,\tau)\|\eta_{\vec{u}_f}^{h,m}\|^2_f
-E(\theta,\tau)\|\eta_{\vec{u}_f}^{h,m-1}\|^2_f\\
&-F(\theta,\tau)
(\|\eta_{\vec{u}_f}^{h,m}\|_f\|\eta_{\vec{u}_f}^{h,m-1}\|_f
-\|\eta_{\vec{u}_f}^{h,m}\|_f\|\eta_{\vec{u}_f}^{h,m-1}\|_f)\bigg),\nonumber
\end{align}
and
 \begin{align}\label{del12}
&2g\big(A(\eta_{\phi_p}^{h,m+1}),B(\eta_{\phi_p}^{h,m+1})\big)_{\Omega_p}\nonumber\\
\geq&
gC_{min}\bigg(D(\theta,\tau)\|\eta_{\phi_p}^{h,m+1}\|^2_p
-H(\theta,\tau)\|\eta_{\phi_p}^{h,m}\|^2_p
-E(\theta,\tau)\|\eta_{\phi_p}^{h,m-1}\|^2_p\\
&-F(\theta,\tau)(\|\eta_{\phi_p}^{h,m}\|_p\|\eta_{\phi_p}^{h,m-1}\|_p
-\|\eta_{\phi_p}^{h,m}\|_p\|\eta_{\phi_p}^{h,m-1}\|_p)\bigg),\nonumber
\end{align}
and
\begin{align}\label{del21}
2k_m a_{\Omega_f}\big(B(\eta_{\vec{u}_f}^{h,m+1}),B(\eta_{\vec{u}_f}^{h,m+1})\big)
\geq 2\hat{C}_{coe}k_m\|B(\eta_{\vec{u}_f}^{h,m+1})\|_{X_f}^2,
\end{align}
and
\begin{align}\label{del22}
2k_m a_{\Omega_p}\big(B(\eta_{\phi_p}^{h,m+1}),B(\eta_{\phi_p}^{h,m+1})\big)
\geq 2g\check{C}_{coe}k_m\|B(\eta_{\phi_p}^{h,m+1})\|_{X_p}^2,
\end{align}
and
\begin{align}\label{der11}
&2k_m\big(\frac{1}{k_m}A(\vec{u}_{f}^{m+1})-B(\vec{u}_{f,t}^{m+1}),
B(\eta_{\vec{u}_f}^{h,m+1})\big)_{\Omega_f}\nonumber\\
\leq&\frac{\hat{C}_{coe}k_m}{6}\|B(\eta_{\vec{u}_f}^{h,m+1})\|_{X_f}^2
+\frac{2(k_m+k_{m-1})^4}{\hat{C}_{coe}}\bigg[\frac{(1-\theta)^2\tau_{m-1}^4+((1-2\theta)\tau_{m-1}+1)^2}{4}\\
    &+((1-\theta)(1-2\theta)\tau_{m-1}-\theta)^2+((1-\theta)(1-2\theta)\tau_{m-1}^2)^2\bigg]
      \int_{t^{m-1}}^{t^{m+1}}\|\vec{\ul}_{ttt}\|^2_{X'_f}dt,\nonumber
\end{align}
and
\begin{align}\label{der12}
&2gk_m\big(\frac{1}{k_m}A(\phi_{p}^{m+1})-B(\phi_{p,t}^{m+1}),
B(\phi_{p}^{m+1})\big)_{\Omega_p}\nonumber\\
\leq&\frac{g\check{C}_{coe}k_m}{6}\|B(\phi_{p}^{m+1})\|_{X_p}^2
+\frac{2g(k_m+k_{m-1})^4}{\check{C}_{coe}}
\bigg[\frac{(1-\theta)^2\tau_{m-1}^4+((1-2\theta)\tau_{m-1}+1)^2}{4}\\
&+((1-\theta)(1-2\theta)\tau_{m-1}-\theta)^2+((1-\theta)(1-2\theta)\tau_{m-1}^2)^2\bigg]
\int_{t^{m-1}}^{t^{m+1}}\|\phi_{p,ttt}\|^2_{X'_p}dt,\nonumber
\end{align}
and
\begin{align}
&2k_m\big(\frac{1}{k_m}A(\zeta_{\vec{u}_{f}}^{h,m+1}),
B(\eta_{\vec{u}_f}^{h,m+1})\big)_{\Omega_f}\\
\leq&\frac{\hat{C}_{coe}k_m}{6}\|B(\eta_{\vec{u}_f}^{h,m+1})\|_{X_f}^2
+\frac{6((1-2\theta)\theta+1)^2+6(1-2\theta)^2\theta^4}{\hat{C}_{coe}}
\int_{t^{m-1}}^{t^{m+1}}\|(P_h^{\vec{u}_f}-I)\vec{u}_{f,t}\|^2_{X_f'}dt,\nonumber
\end{align}
and
\begin{align}
&2gk_m\big(\frac{1}{k_m}A(\zeta_{\phi_{p}}^{h,m+1}),
B(\eta_{\phi_p}^{h,m+1})\big)_{\Omega_p}\\
\leq&\frac{g\check{C}_{coe}k_m}{6}\|B(\eta_{\phi_p}^{h,m+1})\|_{X_p}^2
+\frac{6g((1-2\theta)\theta+1)^2+6g(1-2\theta)^2\theta^4}{\check{C}_{coe}}
\int_{t^{m-1}}^{t^{m+1}}\|(P_h^{\phi_p}-I)\phi_{p,t}\|^2_{X_p'}dt,\nonumber
\end{align}
and
\begin{align}\label{der31}
&2k_m\big<S(\vec{g}_{f}^{m+1}),B(\eta_{\vec{u}_f}^{h,m+1})\big>_{\Omega_f}\\
\leq&\frac{\hat{C}_{coe}k_m}{6}\|B(\eta_{\vec{u}_f}^{h,m+1})\|_{X_f}^2
+\frac{2(k_m+k_{m-1})^4(1-\theta)^2(1-2\theta)^2(\tau_{m-1}^2+1)\tau_{m-1}^2}
{\hat{C}_{coe}(\tau_{m-1}+1)^2}\int_{t^{m-1}}^{t^{m+1}} \|\vec{g}_{f,tt}\|_{X_f'}^2dt,\nonumber
\end{align}
and
\begin{align}\label{der32}
&2gk_m\big<S(g_{p}^{m+1}),B(\eta_{\phi_p}^{h,m+1})\big>_{\Omega_p}\\
\leq&\frac{g\check{C}_{coe}k_m}{6}\|B(\eta_{\phi_p}^{h,m+1})\|_{X_p}^2
+\frac{2g(k_m+k_{m-1})^4(1-\theta)^2(1-2\theta)^2(\tau_{m-1}^2+1)\tau_{m-1}^2}{\check{C}_{coe}(\tau_{m-1}+1)^2}\int_{t^{m-1}}^{t^{m+1}} \|g_{p,tt}\|_{X_p'}^2dt.\nonumber
\end{align}

Next, we mainly analyze the interface terms at the right-hand side, and through some simply making up terms, the interface terms can be written as
\begin{align}
\label{interfaceterm}
&-2k_mc_{\Gamma}\big(B(\eta_{\vec{u}_f}^{h,m+1}),
(1+(1-\theta)\tau_{m-1})\phi_{p}^{h,m}-(1-\theta)\tau_{m-1}\phi_{p}^{h,m-1}-B(\phi_{p}^{m+1})\big)\nonumber\\
&+2k_mc_{\Gamma}\big((1+(1-\theta)\tau_{m-1})\vec{u}_{f}^{h,m}-(1-\theta)\tau_{m-1}\vec{u}_{f}^{h,m-1}
-B(\vec{u}_{f}^{m+1}),B(\eta_{\phi_p}^{h,m+1})\big)\nonumber\\
=&-2k_mc_{\Gamma}\big(B(\eta_{\vec{u}_f}^{h,m+1}),B(\eta_{\phi_p}^{h,m+1})\big)
+2k_mc_{\Gamma}\big(B(\eta_{\vec{u}_f}^{h,m+1}),B(\eta_{\phi_p}^{h,m+1})\big)\\
&-2k_mc_{\Gamma}\big(B(\eta_{\vec{u}_f}^{h,m+1}),W(\eta_{\phi_p}^{h,m+1})\big)
+2k_mc_{\Gamma}\big(W(\eta_{\vec{u}_f}^{h,m+1}),B(\eta_{\phi_p}^{h,m+1}\big)\nonumber\\
&-2k_mc_{\Gamma}\big(B(\eta_{\vec{u}_f}^{h,m+1}),B(\zeta_{\phi_{p}}^{h,m+1})\big)
+2k_mc_{\Gamma}\big(B(\zeta_{\vec{u}_{f}}^{h,m+1}),B(\eta_{\phi_p}^{h,m+1})\big)\nonumber\\
&-2k_mc_{\Gamma}\big(B(\eta_{\vec{u}_f}^{h,m+1}),W(\zeta_{\phi_{p}}^{h,m+1})\big)
+2k_mc_{\Gamma}\big(W(\zeta_{\vec{u}_{f}}^{h,m+1}),B(\eta_{\phi_p}^{h,m+1})\big)\nonumber\\
&-2k_mc_{\Gamma}\big(B(\eta_{\vec{u}_f}^{h,m+1}),W(\phi_{p}^{m+1})\big)
+2k_mc_{\Gamma}\big(W(\vec{u}_{f}^{m+1}),B(\eta_{\phi_p}^{h,m+1})\big).\nonumber
\end{align}

For the first four terms in (\ref{interfaceterm}), they add up to zero by using perpority (\ref{agamma}).
For the next four terms, we combine (\ref{agamma}), Lemma \ref{gamma}
and taking  the appropriate $\varepsilon_3=\varepsilon_5=\frac{3}{\hat{C}_{coe}}$ and $\varepsilon_4=\varepsilon_6=\frac{3g}{\check{C}_{coe}}$ into them, we have
\begin{align}
&-2k_mc_{\Gamma}\big(B(\eta_{\vec{u}_f}^{h,m+1}),B(\zeta_{\phi_{p}}^{h,m+1})\big)
+2k_mc_{\Gamma}\big(B(\zeta_{\vec{u}_{f}}^{h,m+1}),B(\eta_{\phi_p}^{h,m+1})\big)\nonumber\\
&-2k_mc_{\Gamma}\big(B(\eta_{\vec{u}_f}^{h,m+1}),W(\zeta_{\phi_{p}}^{h,m+1})\big)
+2k_mc_{\Gamma}\big(W(\zeta_{\vec{u}_{f}}^{h,m+1}),B(\eta_{\phi_p}^{h,m+1})\big)\nonumber\\
\leq&
\frac{\hat{C}_{coe}k_m}{6}\|B(\eta_{\vec{u}_f}^{h,m+1})\|_{X_f}^2
+\frac{6C_3k_m}{\hat{C}_{coe}h}\|B(\zeta_{\phi_{p}}^{h,m+1})\|_{p}^2
+\frac{g\check{C}_{coe}k_m}{6}\|B(\eta_{\phi_p}^{h,m+1})\|_{X_p}^2
+\frac{6gC_4k_m}{\check{C}_{coe}h}\|B(\zeta_{\vec{u}_{f}}^{h,m+1})\|_{f}^2\nonumber\\
&+\frac{\hat{C}_{coe}k_m}{6}\|B(\eta_{\vec{u}_f}^{h,m+1})\|_{X_f}^2
+\frac{6C_5k_m}{\hat{C}_{coe}h}\|W(\zeta_{\phi_{p}}^{h,m+1})\|_{p}^2
+\frac{g\check{C}_{coe}k_m}{6}\|B(\eta_{\phi_p}^{h,m+1})\|_{X_p}^2
+\frac{6gC_6k_m}{\check{C}_{coe}h}\|W(\zeta_{\vec{u}_{f}}^{h,m+1})\|_{f}^2\nonumber\\
\leq&
\frac{\hat{C}_{coe}k_m}{3}\|B(\eta_{\vec{u}_f}^{h,m+1})\|_{X_f}^2
+\frac{g\check{C}_{coe}k_m}{3}\|B(\eta_{\phi_p}^{h,m+1})\|_{X_p}^2\nonumber\\
&+\frac{6gk_m}{\hat{C}_{coe}}\|B(\zeta_{\phi_{p}}^{h,m+1})\|_{p}^2
+\frac{6k_m}{\check{C}_{coe}}\|B(\zeta_{\vec{u}_{f}}^{h,m+1})\|_{f}^2
+\frac{6gk_m}{\hat{C}_{coe}}\|W(\zeta_{\phi_{p}}^{h,m+1})\|_{p}^2
+\frac{6k_m}{\check{C}_{coe}}\|W(\zeta_{\vec{u}_{f}}^{h,m+1})\|_{f}^2,\nonumber
\end{align}
where the last inequality follows from properly chosen constant $C_3,C_4,C_5$ and $C_6$.

For the last four terms, we use the Taylor expansion with the integral remainder
\begin{align*}
   &\phi_p^{m+1}=\phi_p^{m}+k_m\phi_{p,t}^{m}+\int_{t^{m}}^{t^{m+1}}(t^{m+1}-t)\phi_{p,tt}dt,\\
   &\phi_p^{m-1}=\phi_p^{m}-k_{m-1}\phi_{p,t}^{m}+\int_{t^{m}}^{t^{m-1}}(t^{m-1}-t)\phi_{p,tt}dt,
\end{align*}
then
\begin{align*}
W(\phi_{p}^{m+1})
=&\frac{2(1-\theta)^2\tau_{m-1}^2+(1-\theta)\tau_{m-1}}{\tau_{m-1}+1}\int_{t^{m-1}}^{t^{m}}(t^{m+1}-t)\phi_{p,tt}dt\\
&-\frac{2(1-\theta)^2\tau_{m-1}+1-\theta}{\tau_{m-1}+1}\int_{t^{m}}^{t^{m+1}}(t^{m-1}-t)\phi_{p,tt}dt,
\end{align*}
similarly, we have
\begin{align*}
W(\vec{u}_{f}^{m+1})
=&\frac{2(1-\theta)^2\tau_{m-1}^2+(1-\theta)\tau_{m-1}}{\tau_{m-1}+1}\int_{t^{m-1}}^{t^{m}}(t^{m+1}-t)\vec{u}_{f,tt}dt\\
&-\frac{2(1-\theta)^2\tau_{m-1}+1-\theta}{\tau_{m-1}+1}\int_{t^{m}}^{t^{m+1}}(t^{m-1}-t)\vec{u}_{f,tt}dt.
\end{align*}
By using Lemma \ref{gamma},
and taking  $\varepsilon_7=\frac{3}{\hat{C}_{coe}}$ and $\varepsilon_8=\frac{3g}{\check{C}_{coe}}$ into them,
we get
\begin{align*}
&-2k_mc_{\Gamma}\big(B(\eta_{\vec{u}_f}^{h,m+1}),W(\phi_{p}^{m+1})\big)
+2k_mc_{\Gamma}\big(W(\vec{u}_{f}^{m+1}),B(\eta_{\phi_p}^{h,m+1})\big)\\
\leq&
\frac{\hat{C}_{coe}k_m}{6}\|B(\eta_{\vec{u}_f}^{h,m+1})\|_{X_f}^2
+\frac{6C_3k_m}{\hat{C}_{coe}h}\|W(\phi_{p}^{m+1})\|_{p}^2
+\frac{g\check{C}_{coe}k_m}{6}\|B(\eta_{\phi_p}^{h,m+1})\|_{X_p}^2
+\frac{6gC_4k_m}{\check{C}_{coe}h}\big\|W(\vec{u}_{f}^{m+1})\|_{f}^2\\
\leq&
\frac{\hat{C}_{coe}k_m}{6}\|B(\eta_{\vec{u}_f}^{h,m+1})\|_{X_f}^2
+\frac{g\check{C}_{coe}k_m}{6}\|B(\eta_{\phi_p}^{h,m+1})\|_{X_p}^2\nonumber\\
&+\frac{2(k_m+k_{m-1})^4(2(1-\theta)\tau_{m-1}+1)^2(\tau_{m-1}^2+1)}{\check{C}_{coe}(\tau_{m-1}+1)^2}
\int_{t^{m-1}}^{t^{m+1}}\|\vec{u}_{f,tt}\|_f^2dt\nonumber\\
&+\frac{2g(k_m+k_{m-1})^4(2(1-\theta)\tau_{m-1}+1)^2(\tau_{m-1}^2+1)}{\hat{C}_{coe}(\tau_{m-1}+1)^2}
\int_{t^{m-1}}^{t^{m+1}}\|\phi_{p,tt}\|_p^2dt,\nonumber
\end{align*}
where the inequality follows from properly chosen constant $C_7$ and $C_8$.

Combining the above analysis, the interface terms can be handled as
\begin{align}\label{der4}
&-2k_mc_{\Gamma}\big(B(\eta_{\vec{u}_f}^{h,m+1}),
(1+(1-\theta)\tau_{m-1})\phi_{p}^{h,m}-(1-\theta)\tau_{m-1}\phi_{p}^{h,m-1}-B(\phi_{p}^{m+1})\big)\nonumber\\
&+2k_mc_{\Gamma}\big((1+(1-\theta)\tau_{m-1})\vec{u}_{f}^{h,m}-(1-\theta)\tau_{m-1}\vec{u}_{f}^{h,m-1}
-B(\vec{u}_{f}^{m+1}),B(\eta_{\phi_p}^{h,m+1})\big)\\
\leq&
\frac{\hat{C}_{coe}k_m}{2}\big\|B(\eta_{\vec{u}_f}^{h,m+1})\big\|_{X_f}^2
+\frac{g\check{C}_{coe}k_m}{2}\big\|B(\eta_{\phi_p}^{h,m+1})\big\|_{X_p}^2\nonumber\\
&+\frac{6gk_m}{\hat{C}_{coe}}\big\|B(\zeta_{\phi_{p}}^{h,m+1})\big\|_{p}^2
+\frac{6k_m}{\check{C}_{coe}}\big\|B(\zeta_{\vec{u}_{f}}^{h,m+1})\big\|_{f}^2
+\frac{6gk_m}{\hat{C}_{coe}}\big\|W(\zeta_{\phi_{p}}^{h,m+1})\big\|_{p}^2
+\frac{6k_m}{\check{C}_{coe}}\big\|W(\zeta_{\vec{u}_{f}}^{h,m+1})\big\|_{f}^2\nonumber\\
&+\frac{2(k_m+k_{m-1})^4(2(1-\theta)\tau_{m-1}+1)^2(\tau_{m-1}^2+1)}{\check{C}_{coe}(\tau_{m-1}+1)^2}
\int_{t^{m-1}}^{t^{m+1}}\|\vec{u}_{f,tt}\|_f^2dt\nonumber\\
&+\frac{2g(k_m+k_{m-1})^4(2(1-\theta)\tau_{m-1}+1)^2(\tau_{m-1}^2+1)}{\hat{C}_{coe}(\tau_{m-1}+1)^2}
\int_{t^{m-1}}^{t^{m+1}}\|\phi_{p,tt}\|_p^2dt.\nonumber
\end{align}

Combining the (\ref{del11})-(\ref{der4})  and sum the (\ref{de}) over $m=1,2,...,N-1$, and using the same method as Theorem \ref{theorem3}. Let $\tilde{k}=\max \limits_{1\leq m\leq N-1}\{k_m+k_{m-1}\}$, we have
\begin{align*}
&C_{min}I(\theta,\tau)\|\eta_{\vec{u}_f}^{h,N}\|_f^2
+gC_{min}I(\theta,\tau)\|\eta_{\phi_p}^{h,N}\|_p^2
+\hat{C}_{coe}\sum_{m=1}^{N-1}\big[k_m\big\|B(\eta_{\vec{u}_f}^{h,m+1})\big\|_{X_f}^2\big]
+g\check{C}_{coe}\sum_{m=1}^{N-1}\big[k_m\|b(\eta_{\phi_p}^{h,m+1})\|_{X_p}^2\big]\\
\leq&
\sum_{m=1}^{N-1}\left[
\frac{2\tilde{k}^4}{\hat{C}_{coe}}\bigg(\frac{(1-\theta)^2\tau_{m-1}^4+((1-2\theta)\tau_{m-1}+1)^2}{4}\right.\nonumber\\
    &\left.+((1-\theta)(1-2\theta)\tau_{m-1}-\theta)^2+((1-\theta)(1-2\theta)\tau_{m-1}^2)^2\bigg)
    \int_{t^{m-1}}^{t^{m+1}}\|\vec{u}_{f,ttt}\|_{X_f'}^2dt\right.\\
&\left.+\frac{2g\tilde{k}^4}{\check{C}_{coe}}\bigg(\frac{(1-\theta)^2\tau_{m-1}^4+((1-2\theta)\tau_{m-1}+1)^2}{4}\right.\nonumber\\
    &\left.+((1-\theta)(1-2\theta)\tau_{m-1}-\theta)^2+((1-\theta)(1-2\theta)\tau_{m-1}^2)^2\bigg)
    \int_{t^{m-1}}^{t^{m+1}}\|\phi_{p,ttt}\|_{X_p'}^2dt\right.\\
&\left.+\frac{6((1-2\theta)\tau_{m-1}+1)^2+6(1-2\theta)^2\tau_{m-1}^4}{\hat{C}_{coe}}
\int_{t^{m-1}}^{t^{m+1}}\|(P_h^{\vec{u}_{f}}-I)\vec{u}_{f,t}\|^2_{X_f'}dt\right.\\
&\left.+\frac{6g((1-2\theta)\tau_{m-1}+1)^2+6g(1-2\theta)^2\tau_{m-1}^4}{\check{C}_{coe}}
\int_{t^{m-1}}^{t^{m+1}}\|(P_h^{\phi_p}-I)\phi_{p,t}\|^2_{X_p'}dt\right.\nonumber\\
&\left.+\frac{2\tilde{k}^4(1-\theta)^2(1-2\theta)^2(\tau_{m-1}^2+1)\tau_{m-1}^2}{\hat{C}_{coe}(\tau_{m-1}+1)^2}\int_{t^{m-1}}^{t^{m+1}} \|\vec{g}_{f,tt}\|_{X_f'}^2dt\right.\\
&\left.+\frac{2g\tilde{k}^4(1-\theta)^2(1-2\theta)^2(\tau_{m-1}^2+1)\tau_{m-1}^2}{\check{C}_{coe}(\tau_{m-1}+1)^2}\int_{t^{m-1}}^{t^{m+1}} \|g_{p,tt}\|_{X_p'}^2dt\right.\\
&\left.+\frac{2\tilde{k}^4(2(1-\theta)\tau_{m-1}+1)^2(\tau_{m-1}^2+1)}{\check{C}_{coe}(\tau_{m-1}+1)^2}
\int_{t^{m-1}}^{t^{m+1}}\|\vec{u}_{f,tt}\|_f^2dt\right.\\
&\left.+\frac{2g\tilde{k}^4(2(1-\theta)\tau_{m-1}+1)^2(\tau_{m-1}^2+1)}{\hat{C}_{coe}(\tau_{m-1}+1)^2}
\int_{t^{m-1}}^{t^{m+1}}\|\phi_{p,tt}\|_p^2dt\right.\\
&\left.+\frac{6g\tilde{k}}{\hat{C}_{coe}}\|B(\zeta_{\phi_{p}}^{h,m+1})\|_{p}^2
+\frac{6\tilde{k}}{\check{C}_{coe}}\|B(\zeta_{\vec{u}_{f}}^{h,m+1})\|_{f}^2
+\frac{6g\tilde{k}}{\hat{C}_{coe}}\|W(\zeta_{\phi_{p}}^{h,m+1})\|_{p}^2
+\frac{6\tilde{k}}{\check{C}_{coe}}\|W(\zeta_{\vec{u}_{f}}^{h,m+1})\|_{f}^2\right].
\end{align*}

Finally, using the triangle inequality, we end the proof.
\end{proof}
\section{Numerical experiments\label{5}}
In this section, we do two numerical experiments.
In the first test, we verify the effectiveness of the coupled and decoupled variable time-stepping algorithms by three different sets of variation rules for time steps $k_m^{1}$, $k_m^{2}$, $k_m^{3}$.
In the second test, we use the adaptive algorithm to verify the second-order convergence.
The following numerical experiments are implemented using the Software package FreeFEM++, and we set all the physical parameters $n$, $\rho$, $g$, $\nu$, $\mathds{K}$, $S_0$ and $\alpha$ are equal to 1, and the initial conditions, boundary conditions and the source terms follow from the exact solution. We use the well-known Tayor-Hood elements(P2-P1) for the fluid equations and the continuous piecewise quadratic elements(P2) for the porous media flow equation.
\begin{figure}[!ht]
\centering
\subfigure[Coupled algorithm with $\theta=0.2$]
{
\begin{minipage}{0.3\linewidth}
\centering
\includegraphics[width=0.9\linewidth]{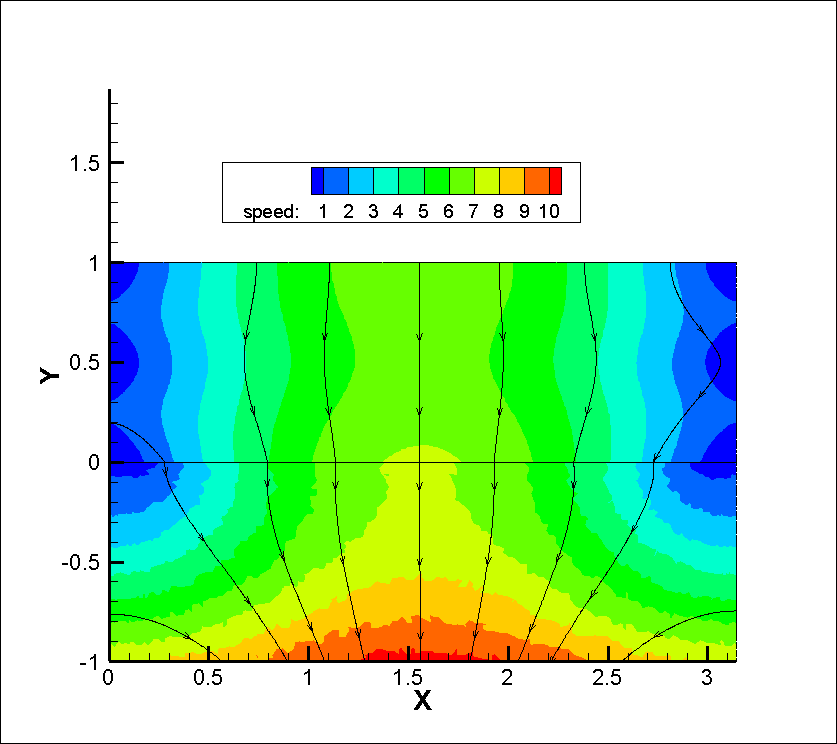}
\end{minipage}
}
\subfigure[Coupled algorithm with $\theta=0.3$]
{
\begin{minipage}{0.3\linewidth}
\centering
\includegraphics[width=0.9\linewidth]{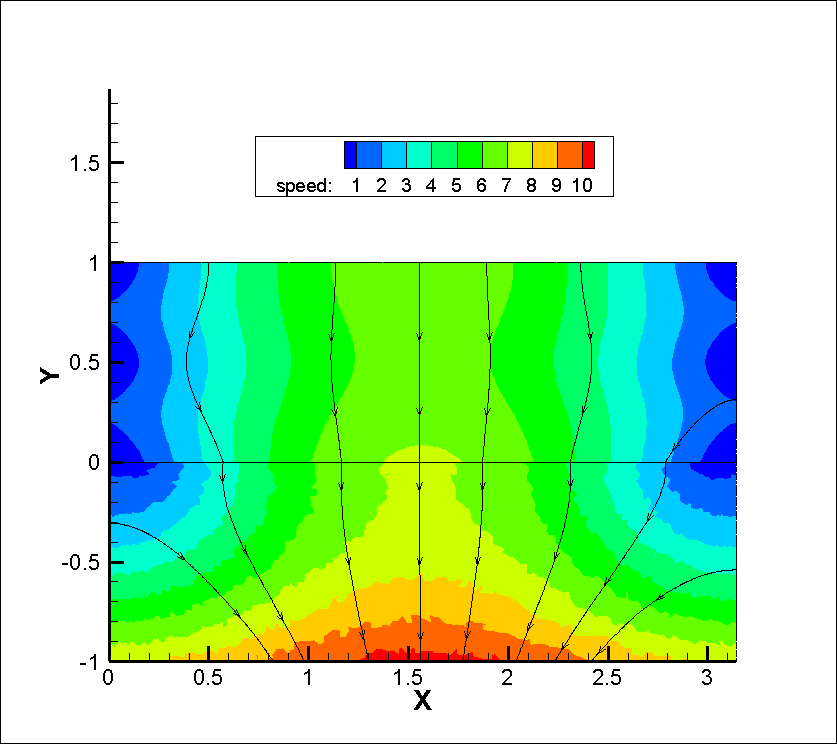}
\end{minipage}
}
\subfigure[Coupled algorithm with $\theta=0.4$]
{
\begin{minipage}{0.3\linewidth}
\centering
\includegraphics[width=0.9\linewidth]{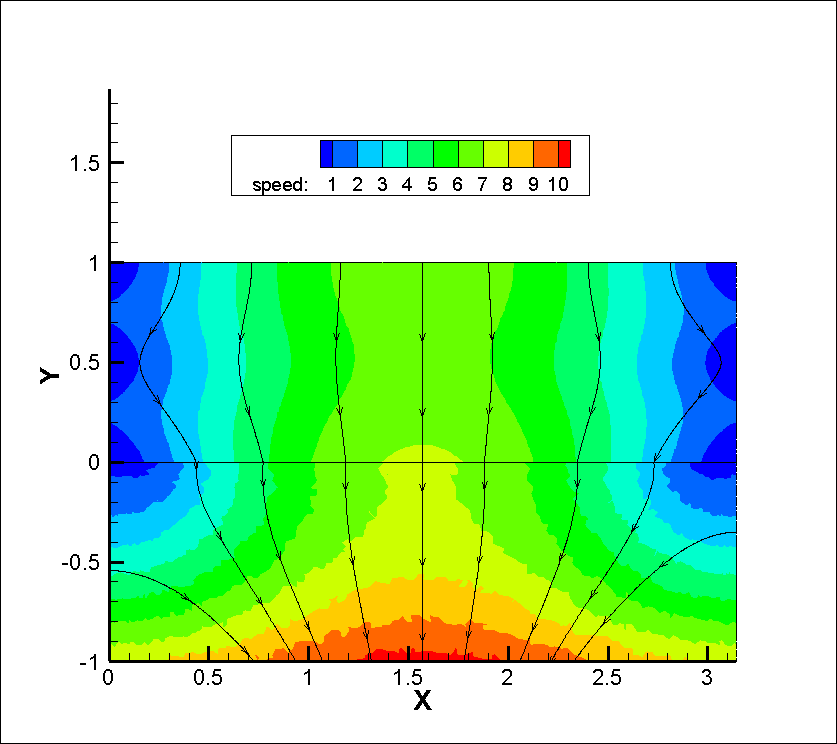}
\end{minipage}
}
\subfigure[Decoupled scheme with $\theta=0.2$]
{
\begin{minipage}{0.3\linewidth}
\centering
\includegraphics[width=0.9\linewidth]{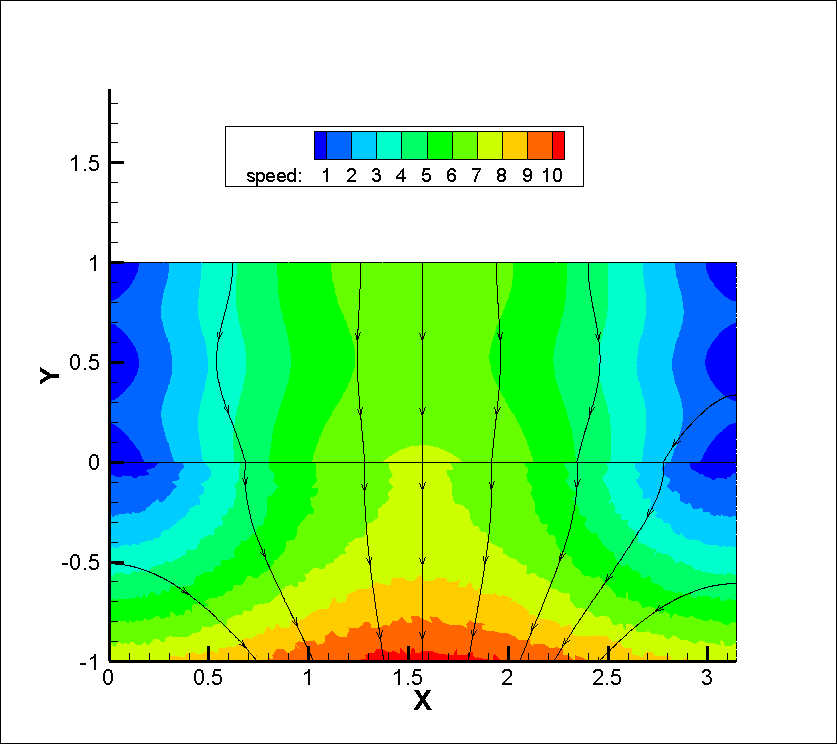}
\end{minipage}
}
\subfigure[Decoupled algorithm with $\theta=0.3$]
{
\begin{minipage}{0.3\linewidth}
\centering
\includegraphics[width=0.9\linewidth]{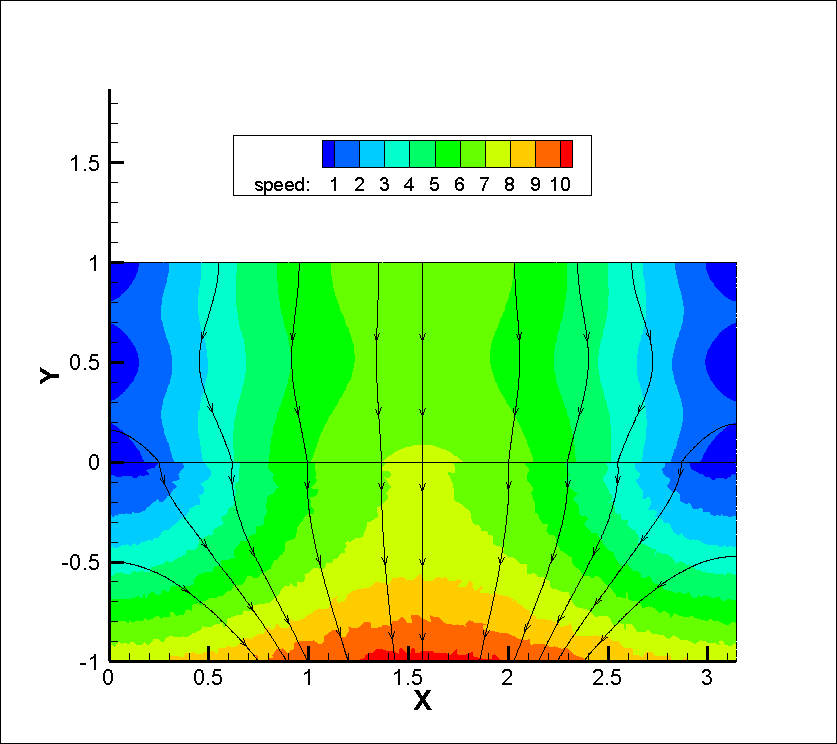}
\end{minipage}
}
\subfigure[Decoupled algorithm with $\theta=0.4$]
{
\begin{minipage}{0.3\linewidth}
\centering
\includegraphics[width=0.9\linewidth]{km1decoupled4}
\end{minipage}
}
\caption{Speed contours and velocity streamlines for Linear Multistep method plus time filter with $k_m^{1}$.}
\label{fig2}
\end{figure}

\begin{figure}[!ht]
\centering
\subfigure[Coupled algorithm with $\theta=0.2$]
{
\begin{minipage}{0.3\linewidth}
\centering
\includegraphics[width=0.9\linewidth]{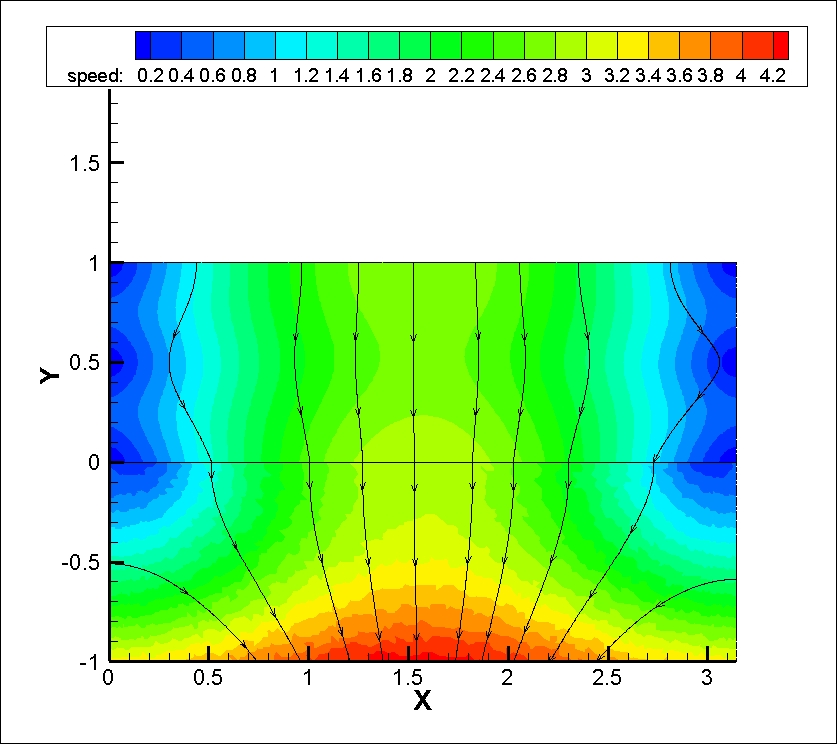}
\end{minipage}
}
\subfigure[Coupled algorithm with $\theta=0.3$]
{
\begin{minipage}{0.3\linewidth}
\centering
\includegraphics[width=0.9\linewidth]{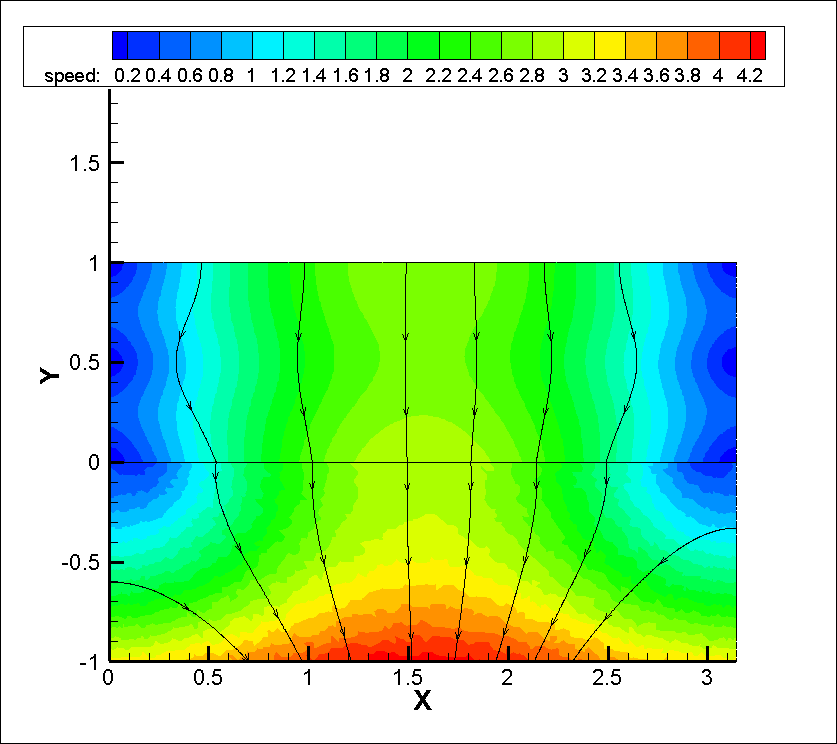}
\end{minipage}
}
\subfigure[Coupled scheme with $\theta=0.4$]
{
\begin{minipage}{0.3\linewidth}
\centering
\includegraphics[width=0.9\linewidth]{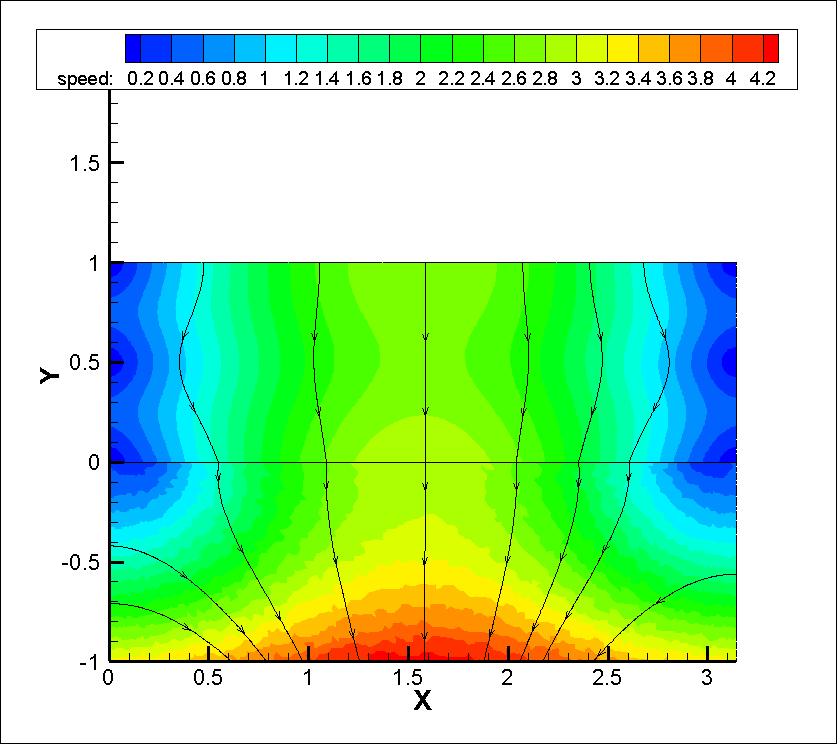}
\end{minipage}
}
\subfigure[Decoupled algorithm with $\theta=0.2$]
{
\begin{minipage}{0.3\linewidth}
\centering
\includegraphics[width=0.9\linewidth]{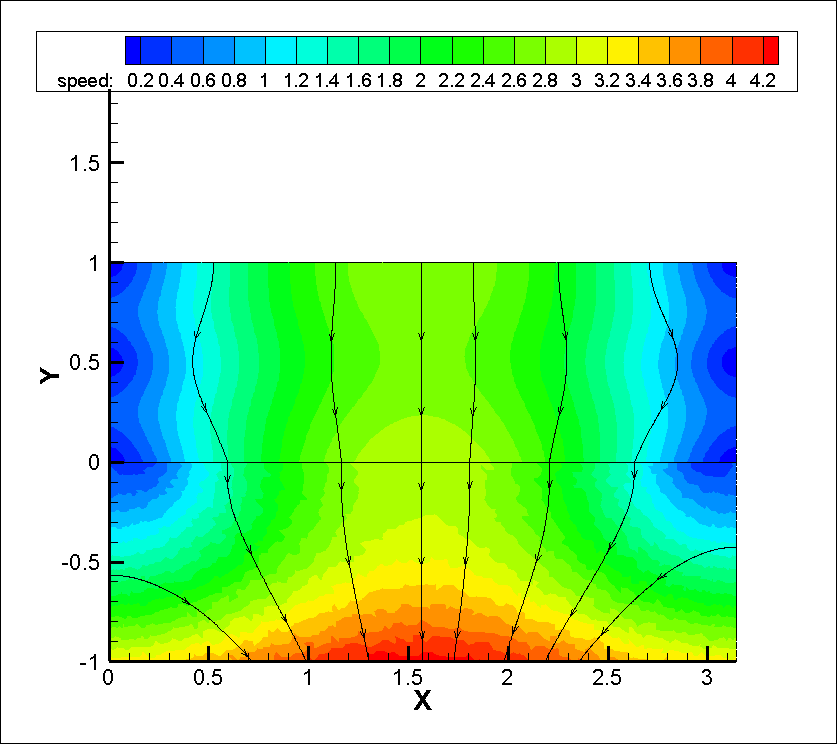}
\end{minipage}
}
\subfigure[Decoupled algorithm with $\theta=0.3$]
{
\begin{minipage}{0.3\linewidth}
\centering
\includegraphics[width=0.9\linewidth]{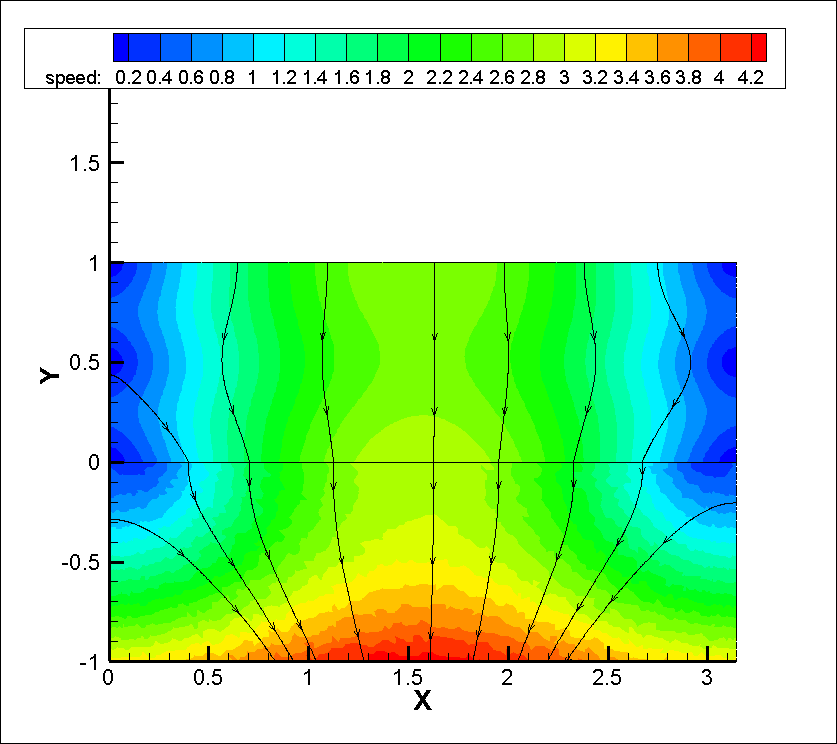}
\end{minipage}
}
\subfigure[Decoupled algorithm with $\theta=0.4$]
{
\begin{minipage}{0.3\linewidth}
\centering
\includegraphics[width=0.9\linewidth]{km2decoupled4}
\end{minipage}
}
\caption{Speed contours and velocity streamlines for Linear Multistep method plus time filter with $k_m^{2}$.}
\label{fig3}
\end{figure}

\begin{figure}[!ht]
\centering
\subfigure[Coupled algorithm with $\theta=0.2$]
{
\begin{minipage}{0.3\linewidth}
\centering
\includegraphics[width=0.9\linewidth]{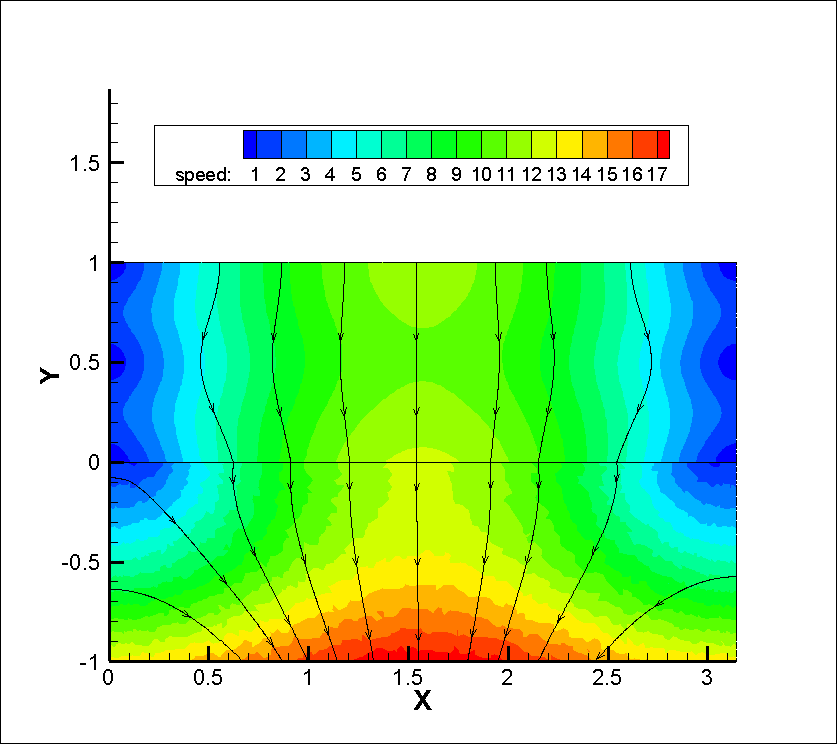}
\end{minipage}
}
\subfigure[Coupled algorithm with $\theta=0.3$]
{
\begin{minipage}{0.3\linewidth}
\centering
\includegraphics[width=0.9\linewidth]{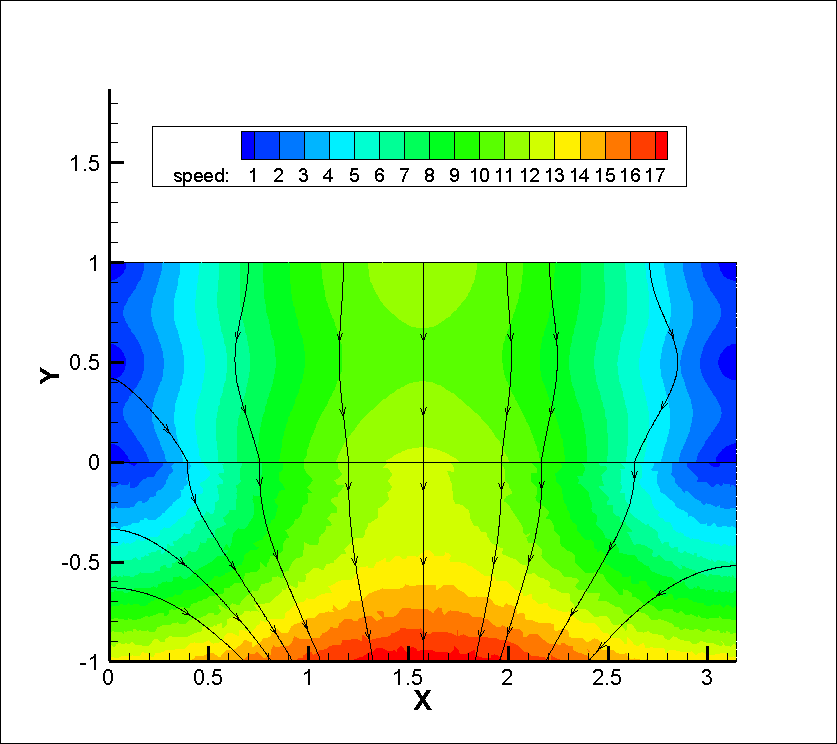}
\end{minipage}
}
\subfigure[Coupled algorithm with $\theta=0.4$]
{
\begin{minipage}{0.3\linewidth}
\centering
\includegraphics[width=0.9\linewidth]{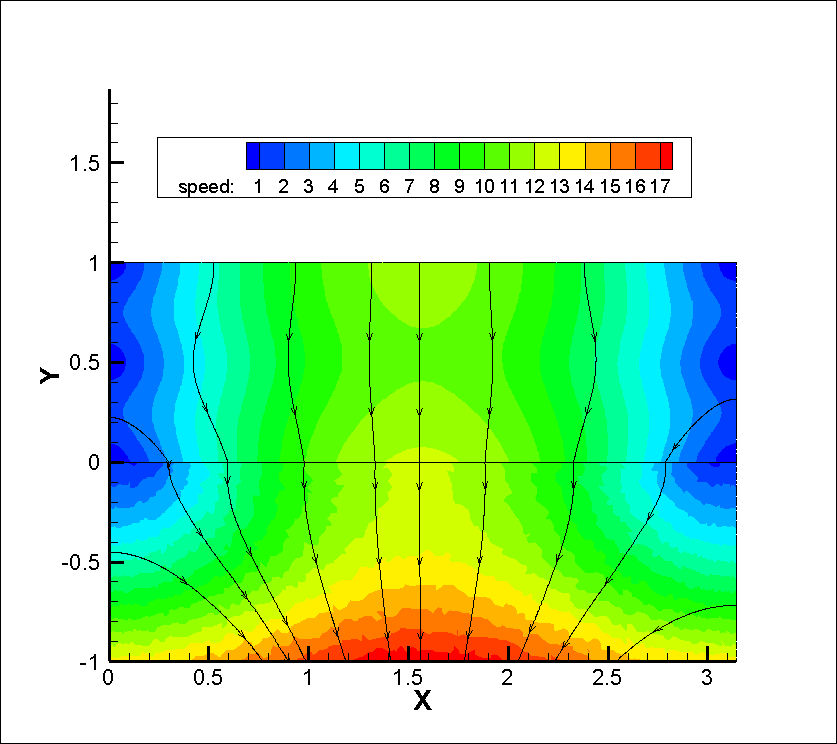}
\end{minipage}
}
\subfigure[Decoupled algorithm with $\theta=0.2$]
{
\begin{minipage}{0.3\linewidth}
\centering
\includegraphics[width=0.9\linewidth]{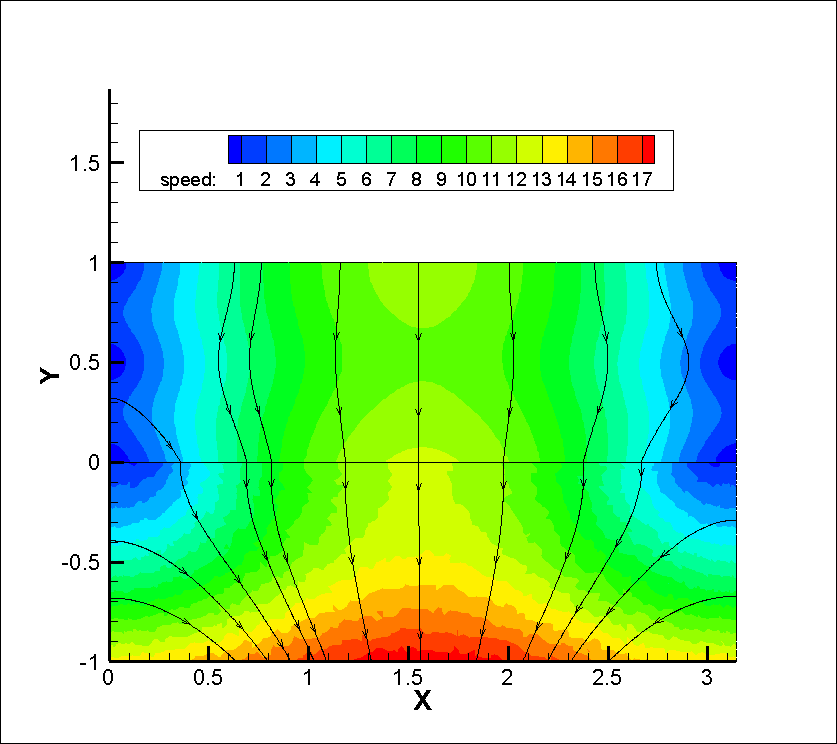}
\end{minipage}
}
\subfigure[Decoupled algorithm with $\theta=0.3$]
{
\begin{minipage}{0.3\linewidth}
\centering
\includegraphics[width=0.9\linewidth]{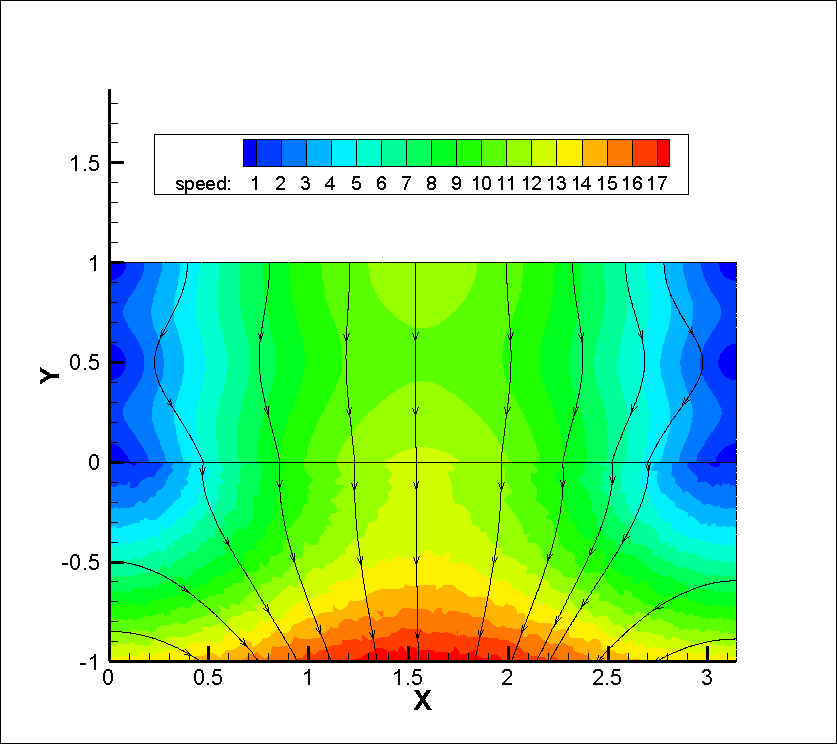}
\end{minipage}
}
\subfigure[Decoupled algorithm with $\theta=0.4$]
{
\begin{minipage}{0.3\linewidth}
\centering
\includegraphics[width=0.9\linewidth]{km3decoupled4}
\end{minipage}
}
\caption{Speed contours and velocity streamlines for Linear Multistep method plus time filter with $k_m^{3}$.}
\label{fig4}
\end{figure}
\subsection{Test of the effectiveness for the variable time-stepping algorithms}
Here we use the numerical test from \cite{AEEA}, let the computational domain $\Omega$ be composed of $\Omega_f=(0,\pi)\times(0,1)$ and $\Omega_p=(0,\pi)\times(-1,0)$ with the interface $\Gamma=(0,\pi)\times{0}$. The exact solution is given by
\begin{eqnarray*}
&&\vec{u}_f=[\frac{1}{\pi}sin(2\pi y)cos(x)e^t, (-2+\frac{1}{\pi^2}sin^2(\pi y))sin(x)e^t],\\
&&p_f=0,\\
&&\phi_p=(e^y-e^{-y})sin(x)e^t.
\end{eqnarray*}

For this test, we change the time step size to observe the effect on the experimental results and set the diameters $h=1/100$.
We use coupled and decoupled algorithms to this test problem for 40 time steps and refer to the time step size $k_m^{1}$, $k_m^{2}$ and $k_m^{3}$ similar to that in \cite{AVAC}:
\begin{equation*}
k_m^{1}=
0.01+0.05t_m,\qquad m\geq0,
\end{equation*}
and
\begin{equation*}
k_m^{2}=\left\{
\begin{aligned}
&0.01,&\qquad 0\leq m\leq 10,\\
&0.01+0.05sin(10t_m),&\qquad m>10,
\end{aligned}
\right .
\end{equation*}
and
\begin{equation*}
k_m^{3}=
0.1-0.05t_m,\qquad m\geq0.
\end{equation*}

Figure \ref{fig2}-Figure \ref{fig4} show speed contours and velocity streamlines of coupled and decoupled Linear Multi-step methods plus time filter for $\theta=0.2$, $0.3$, $0.4$  with different time step size $k_m^{1}$, $k_m^{2}$ and $k_m^{3}$, respectively. From these figures, we can see that these variable time-stepping algorithms can effectively simulate fluid motion regardless of whether the time step increases or decreases.
\subsection{Test of the convergence and efficiency for the coupled and decoupled adaptive algorithms}
Here we use the numerical test from \cite{MMXZ}, considering the model problem on $\Omega_f=(0,1)\times(1,2)$ and $\Omega_p=(0,1)\times(0,1)$ with the interface $\Gamma=(0,1)\times{1}$. The exact solution is:
\begin{eqnarray*}
&&\vec{u}_f=((x^2(y-1)^2+y)cos(t),-\frac{2}{3}x(y-1)^3cos(t)+(2-\pi\sin(\pi x))cos(t)),\\
&&p_f=(2-\pi\sin(\pi x))\sin(\frac{1}{2}\pi y)cos(t),\\
&&\phi_p=(2-\pi\sin(\pi x))(1-y-\cos(\pi y))cos(t).
\end{eqnarray*}

In order to demonstrate the convergence and efficiency of the variable time-stepping algorithm for coupled and decoupled Linear Multi-step methods plus time filter, we show the convergence order by results of adaptive algorithm. Here we choose the Linear Multi-step method when $\theta=0.3$ and list the error, convergence order and CPU time in Table \ref{coupledlmt}-Table \ref{decoupledlmtfh}.

We first calculate the convergence order and CPU time of the coupled and decoupled Linear Multi-step methods and the Linear Multi-step methods plus time filter algorithms by varying the time step, and fix the mesh size $h=\frac{1}{120}$ and the final time $T=1.0$. In this experiment, we vary the tolerance $\varepsilon$ from 1e-3 to 1e-6, and use $\bar{\Delta t}$  denote the average time step size.
We define
$$\rho_{\bar{\Delta t},v}=\frac{lg(e_{v}(\bar{\Delta t}_{1})/e_{v}(\bar{\Delta t}_{2}) )}{lg(\bar{\Delta t}_{1}/\bar{\Delta t}_{2})}$$
where $e_{v}(\bar{\Delta t})=\big(\sum\limits^{N}_{i=2} k_{i}\frac{\|v(t_{i})-v^{h,i}\|_{L^2}}
{\|v(t_{i})\|_{L^2}}\big)^{\frac{1}{2}}$ $v=\vec{u}_f$, $p_f$ and $\phi_p$.
Table \ref{coupledlmt} and Table \ref{decoupledlmt} show the error, convergence order and CPU time of the coupled and decoupled adaptive Linear Multi-step method, and Table \ref{coupledlmtft} and Table \ref{decoupledlmtft} show the error, convergence order and CPU time of the coupled and decoupled adaptive Linear Multi-step methods plus time filter algorithms, respectively.
From Table \ref{coupledlmt} and Table \ref{decoupledlmt}, it can be found that the convergence order of the Linear Multi-step method is the first order, and from Table \ref{coupledlmtft} and Table \ref{decoupledlmtft} , we can find that the convergence order of the Linear Multi-step method plus time filter algorithm can reach the second order.

Then we calculate the convergence order and CPU time of the coupled and decoupled Linear Multi-step methods and the Linear Multi-step methods plus time filter by varying the mesh size $h$ with a fixed time step $\Delta t=0.01$ and the final time $T=1.0$. Similarly, we estimate the corresponding convergence order by
$$\rho_{h,v}=\frac{lg(e_{v}(h_1)/e_{v}(h_2))}{lg(h_1/h_2)}$$
where $e_{v}(h)=\|v(t_{i})-v^{h,i}\|_{L^2}$ is the error computed by the algorithm with fixed time step $\Delta t$. We change the mesh size $h$ from $\frac{1}{4}$ to $\frac{1}{64}$, and get error, convergence order and CPU time of coupled and decoupled Linear Multi-step methods and the Linear Multi-step methods plus time filter algorithms in the Table \ref{coupledlmh}-Table \ref{decoupledlmtfh}.
From \ref{coupledlmtfh} and \ref{decoupledlmtfh}, we can clearly state that the coupled and decoupled Linear Multi-step methods plus time filter are convergent in mesh size $h$ and the order of convergence both are $O(h^2)$.

Finally, comparing the convergence order and CPU time in Table \ref{coupledlmtfh} and Table \ref{decoupledlmtfh}, we can find that the coupled and decoupled Linear Multi-step methods plus time filter algorithms can achieve second-order convergence, but the decoupled Linear Multi-step methods plus time filter algorithm takes less computation time, that is, the decoupled algorithm is more efficient.
%\textbf{Competing interests statement}: The authors have no relevant financial or non-financial interests to disclose.
\begin{table}[!ht]
\caption{The convergence order and CPU time of coupled adaptive algorithm for Linear Multi-step method at time $T=1$, with fixed mesh size $h=\frac{1}{120}$.}
\centering
\begin{tabular}{clclclclclclclcl}
\hline
$\bar{\Delta t}$&
$\big(\sum\limits^{N}_{i=2} k_{i}\frac{\|\vec{\u}_{f}(t_{i})-\vec{u}_{f}^{h,i}\|_{f}}
{\|\vec{\u}_{f}(t_{i})\|_{f}}\big)^{\frac{1}{2}}$&
$\rho_{\u_f}$&
$\big(\sum\limits_{i=2}^{N}k_{i}\frac{\|{p}_{f}(t_{i})-{p}_{f}^{h,i}\|_{f}}
{\|{p}_{f}(t_{i})\|_{f}}\big)^{\frac{1}{2}}$&
$\rho_{p_f}$&
$\big(\sum\limits_{i=2}^{N}k_{i}\frac{\|{\phi}_{f}(t_{i})-{phi}_{f}^{h,i}\|_{f}}
{\|{\phi}_{f}(t_{i})\|_{f}}\big)^{\frac{1}{2}}$&
$\rho_{\phi_p}$&
$CPU(s)$&\\
\hline
$\frac{1}{35}$ &7.42414e-05&-   &0.0116894&-   &0.000269309&-      &892.42\\
$\frac{1}{89}$ &2.33383e-05&1.24&0.003654 &1.25&8.50719e-05&1.23&2465.76\\
$\frac{1}{259}$&7.57547e-06&1.05&0.001212 &1.03&2.76377e-05&1.05&7157.29\\
$\frac{1}{803}$&2.39375e-06&1.02&0.0003901&1.01&8.74451e-06&1.02&18066.00\\                                                              \hline
\end{tabular}
\label{coupledlmt}
\end{table}
\begin{table}[!ht]
\caption{The convergence order and CPU time of decoupled adaptive algorithm for Linear Multi-step method at time $T=1$, with fixed mesh size $h=\frac{1}{120}$.}
\centering
\begin{tabular}{clclclclclclclcl}
\hline
$\bar{\Delta t}$&
$\big(\sum\limits^{N}_{i=2} k_{i}\frac{\|\vec{\u}_{f}(t_{i})-\vec{u}_{f}^{h,i}\|_{f}}
{\|\vec{\u}_{f}(t_{i})\|_{f}}\big)^{\frac{1}{2}}$&
$\rho_{\u_f}$&
$\big(\sum\limits_{i=2}^{N}k_{i}\frac{\|{p}_{f}(t_{i})-{p}_{f}^{h,i}\|_{f}}
{\|{p}_{f}(t_{i})\|_{f}}\big)^{\frac{1}{2}}$&
$\rho_{p_f}$&
$\big(\sum\limits_{i=2}^{N}k_{i}\frac{\|{\phi}_{f}(t_{i})-{phi}_{f}^{h,i}\|_{f}}
{\|{\phi}_{f}(t_{i})\|_{f}}\big)^{\frac{1}{2}}$&
$\rho_{\phi_p}$&
$CPU(s)$&\\
\hline
$\frac{1}{35}$ &0.000227219&-   &0.0155927  &-   &0.00220716 &-   &594.259\\
$\frac{1}{89}$ &6.81412e-05&1.29&0.00481806 &1.25&0.00067191 &1.27&1679.09\\
$\frac{1}{259}$&2.22322e-05&1.05&0.0016059  &1.03&0.000218926&1.05&4557.1\\
$\frac{1}{803}$&6.97375e-06&1.02&0.000512693&1.01&6.87357e-05&1.02&12039.1\\                                                              \hline
\end{tabular}
\label{decoupledlmt}
\end{table}
\begin{table}[!ht]
\caption{The convergence order and CPU time of coupled adaptive algorithm for Linear Multi-step method plus time filter at time $T=1$, with fixed mesh size $h=\frac{1}{120}$.}
\centering
\begin{tabular}{clclclclclclclcl}
\hline
$\bar{\Delta t}$&
$\big(\sum\limits^{N}_{i=2} k_{i}\frac{\|\vec{\u}_{f}(t_{i})-\vec{u}_{f}^{h,i}\|_{f}}
{\|\vec{\u}_{f}(t_{i})\|_{f}}\big)^{\frac{1}{2}}$&
$\rho_{\u_f}$&
$\big(\sum\limits_{i=2}^{N}k_{i}\frac{\|{p}_{f}(t_{i})-{p}_{f}^{h,i}\|_{f}}
{\|{p}_{f}(t_{i})\|_{f}}\big)^{\frac{1}{2}}$&
$\rho_{p_f}$&
$\big(\sum\limits_{i=2}^{N}k_{i}\frac{\|{\phi}_{f}(t_{i})-{phi}_{f}^{h,i}\|_{f}}
{\|{\phi}_{f}(t_{i})\|_{f}}\big)^{\frac{1}{2}}$&
$\rho_{\phi_p}$&
$CPU(s)$&\\
\hline
$\frac{1}{12}$ &0.00465218 &-   &0.0485196  &-   &0.00429906 &-   &232.856\\
$\frac{1}{33}$ &0.00019342 &2.24&0.00848417 &1.23&0.000174177&2.26&940.701\\
$\frac{1}{107}$&2.04724e-05&1.91&0.00243725 &1.06&1.82232e-05&1.92&3758.53\\
$\frac{1}{292}$&2.32442e-06&2.17&0.000788211&1.12&2.09898e-06&2.15&10744.3\\                                                              \hline
\end{tabular}
\label{coupledlmtft}
\end{table}
\begin{table}[!ht]
\caption{The convergence order and CPU time of decoupled adaptive algorithm for Linear Multi-step method plus time filter at time $T=1$, with fixed mesh size $h=\frac{1}{120}$.}
\centering
\begin{tabular}{clclclclclclclcl}
\hline
$\bar{\Delta t}$&
$\big(\sum\limits^{N}_{i=2} k_{i}\frac{\|\vec{\u}_{f}(t_{i})-\vec{u}_{f}^{h,i}\|_{f}}
{\|\vec{\u}_{f}(t_{i})\|_{f}}\big)^{\frac{1}{2}}$&
$\rho_{\u_f}$&
$\big(\sum\limits_{i=2}^{N}k_{i}\frac{\|{p}_{f}(t_{i})-{p}_{f}^{h,i}\|_{f}}
{\|{p}_{f}(t_{i})\|_{f}}\big)^{\frac{1}{2}}$&
$\rho_{p_f}$&
$\big(\sum\limits_{i=2}^{N}k_{i}\frac{\|{\phi}_{f}(t_{i})-{phi}_{f}^{h,i}\|_{f}}
{\|{\phi}_{f}(t_{i})\|_{f}}\big)^{\frac{1}{2}}$&
$\rho_{\phi_p}$&
$CPU(s)$&\\
\hline
$\frac{1}{12}$ &0.00127755 &-   &0.0215585 &-   &0.00122245 &-   &217.864\\
$\frac{1}{33}$ &0.000194746&1.85&0.0084949 &0.91&0.000183002&2.26&611.952\\
$\frac{1}{107}$&2.05827e-05&1.91&0.00243772&1.06&1.90888e-05&1.92&1487.25\\
$\frac{1}{292}$&2.34181e-06&2.17&0.00078885&1.12&2.20949e-06&2.15&3884.59\\                                                                \hline
\end{tabular}
\label{decoupledlmtft}
\end{table}

\begin{table}[!ht]
\caption{The convergence order of coupled Linear Multi-step method at time $T=1$, with varying mesh size $h$, but fixed time step $\Delta t=0.01$.}
\centering
\begin{tabular}{clclclclclclclcl}
\hline
$h$&
$\|\vec{u}_{f}-\vec{u}_{f}^{h}\|_{L^2}$&
$\rho_{h,\u_f}$&
$\|p_{f}-p_{f}^{h}\|_{L^2}$&
$\rho_{h,p_f}$&
$\|\phi_p-\phi_{p}^{h}\|_{L^2}$&
$\rho_{h,\phi_p}$&
$CPU(s)$&\\
\hline
$\frac{1}{4}$  &0.0697284  &-   &0.353445  &-   &0.0665296  &-   &1.439\\
$\frac{1}{8}$  &0.0176147  &1.98&0.108404  &1.71&0.0185575  &1.84&5.192\\
$\frac{1}{16}$ &0.00441526 &2.00&0.0358228 &1.60&0.00480244 &1.95&21.161\\
$\frac{1}{32}$ &0.00110726 &2.00&0.0123945 &1.53&0.00122367 &1.97&84.602\\
$\frac{1}{64}$ &0.000280437&1.98&0.00463628&1.42&0.000319996&1.93&345.47\\                                                                    \hline
\end{tabular}
\label{coupledlmh}
\end{table}
\begin{table}[!ht]
\caption{The convergence order of decoupled Linear Multi-step method at time $T=1$, with varying mesh size $h$, but fixed time step $\Delta t=0.01$.}
\centering
\begin{tabular}{clclclclclclclcl}
\hline
$h$&
$\|\vec{u}_{f}-\vec{u}_{f}^{h}\|_{L^2}$&
$\rho_{h,\u_f}$&
$\|p_{f}-p_{f}^{h}\|_{L^2}$&
$\rho_{h,p_f}$&
$\|\phi_p-\phi_{p}^{h}\|_{L^2}$&
$\rho_{h,\phi_p}$&
$CPU(s)$&\\
\hline
$\frac{1}{4}$  &0.0697284  &-   &0.353446  &-   &0.0665306  &-   &1.121\\
$\frac{1}{8}$  &0.0176146  &1.98&0.108405  &1.71&0.018559   &1.84&3.823\\
$\frac{1}{16}$ &0.00441515 &2.00&0.0358226 &1.60&0.004804   &1.95&13.928\\
$\frac{1}{32}$ &0.00110714 &2.00&0.0123945 &1.53&0.00122528 &1.97&55.459\\
$\frac{1}{64}$ &0.000280312&1.98&0.00463652&1.42&0.000321683&1.93&231.263\\                                                                   \hline
\end{tabular}
\label{decoupledlfh}
\end{table}
\begin{table}[!ht]
\caption{The convergence order of coupled Linear Multi-step method plus time filter at time $T=1$, with varying mesh size $h$, but fixed time step $\Delta t=0.01$.}
\centering
\begin{tabular}{clclclclclclclcl}
\hline
$h$&
$\|\vec{u}_{f}-\vec{u}_{f}^{h}\|_{L^2}$&
$\rho_{h,\u_f}$&
$\|p_{f}-p_{f}^{h}\|_{L^2}$&
$\rho_{h,p_f}$&
$\|\phi_p-\phi_{p}^{h}\|_{L^2}$&
$\rho_{h,\phi_p}$&
$CPU(s)$&\\
\hline
$\frac{1}{4}$  &0.0697243  &-   &0.353554 &-   &0.0665197  &-   &1.44\\
$\frac{1}{8}$  &0.0176071  &1.99&0.108463 &1.70&0.0185423  &1.84&5.487\\
$\frac{1}{16}$ &0.00440661 &2.00&0.0358539&1.60&0.00478542 &1.95&22.387\\
$\frac{1}{32}$ &0.00109835 &2.00&0.0124059&1.53&0.00120608 &2.00&86.669\\
$\frac{1}{64}$ &0.000271553&2.02&0.0046228&1.42&0.000301909&2.00&352.736\\                                                                   \hline
\end{tabular}
\label{coupledlmtfh}
\end{table}
\begin{table}[!ht]
\caption{The convergence order of decoupled Linear Multi-step method plus time filter at time $T=1$, with varying mesh size $h$ but fixed time step $\Delta t=0.01$.}
\centering
\begin{tabular}{clclclclclclclcl}
\hline
$h$&
$\|\vec{u}_{f}-\vec{u}_{f}^{h}\|_{L^2}$&
$\rho_{h,\u_f}$&
$\|p_{f}-p_{f}^{h}\|_{L^2}$&
$\rho_{h,p_f}$&
$\|\phi_p-\phi_{p}^{h}\|_{L^2}$&
$\rho_{h,\phi_p}$&
$CPU(s)$&\\
\hline
$\frac{1}{4}$  &0.0697243  &-   &0.353555  &-   &0.066521   &-   &1.141\\
$\frac{1}{8}$  &0.017607   &2.00&0.108463  &1.70&0.018544   &1.94&3.605\\
$\frac{1}{16}$ &0.00440648 &2.00&0.0358538 &1.60&0.00478728 &1.95&13.596\\
$\frac{1}{32}$ &0.00109821 &2.00&0.0124061 &1.54&0.00120798 &1.99&56.649\\
$\frac{1}{64}$ &0.000271425&2.02&0.00426237&1.42&0.000303826&1.99&237.072\\                                                                   \hline
\end{tabular}
\label{decoupledlmtfh}
\end{table}

\end{document}